\documentclass[a4paper,12pt,reqno]{amsart}
\usepackage{amsfonts}
\usepackage{mathrsfs}
\usepackage{amsmath,amssymb,latexsym,amsfonts,amscd, yfonts}
\usepackage{cases}
\usepackage{enumitem}
\usepackage[all]{xy}
\allowdisplaybreaks[4]

\newtheorem{thm}{Theorem}[section]
\newtheorem{lem}[thm]{Lemma}
\newtheorem{cor}[thm]{Corollary}

\newtheorem{prop}[thm]{Proposition}
\newtheorem{claim}[thm]{Claim}

\theoremstyle{definition}

\newtheorem{definition}[thm]{Definition}

\newtheorem{say}[thm]{}
\newtheorem{question}[thm]{Question}

\newtheorem{exmple}[thm]{Example}

\theoremstyle{remark}
\newtheorem{remark}[thm]{Remark}

\numberwithin{equation}{section}

\newcommand{\bQ}{{\mathbb Q}}
\newcommand{\bP}{{\mathbb P}}
\newcommand{\roundup}[1]{\lceil{#1}\rceil}
\newcommand{\rounddown}[1]{\lfloor{#1}\rfloor}

\newcommand\lra{\longrightarrow}

\newcommand\OO{{\mathcal{O}}}

\newcommand\FF{{\mathcal{F}}}
\newcommand\HH{{\mathcal{H}}}
\newcommand\GG{{\mathcal{G}}}

\newcommand\Supp{{\text{\rm Supp}}}
\newcommand\lct{{\text{\rm lct}}}
\newcommand\glct{{\text{\rm glct}}}
\newcommand\coeff{{\text{\rm coeff}}}
\newcommand\mult{{\text{\rm mult}}}
\newcommand\Bs{{\text{\rm Bs}}}
\newcommand\vol{\text{\rm vol}}
\newcommand\Mov{{\text{\rm Mov}}}

\newcommand\bF{{\mathbb{F}}}

\newcommand{\q}[0]{{\mathbb Q}}
\newcommand{\cc}[0]{{\mathbb C}}
\newcommand{\GL}{\mathrm{GL}}

\newcommand{\SL}{\mathrm{SL}}
\newcommand{\rdown}[1]{\lfloor{#1}\rfloor}
\newcommand{\onto}[0]{\twoheadrightarrow}
\newcommand{\simr}[0]{\sim_{\r}}
\renewcommand{\r}[0]{{\mathbb R}}
\newcommand{\p}[0]{{\mathbb P}}
\newcommand{\qtq}[1]{\quad\mbox{#1}\quad}

\title{The Noether inequality for algebraic threefolds}
\author{Jungkai A. Chen, Meng Chen, Chen Jiang\\
(W\MakeLowercase{ith an} A\MakeLowercase{ppendix by} J\MakeLowercase{\'{a}nos} K\MakeLowercase{oll\'{a}r})}

\address{\rm National Center for Theoretical Sciences, and Department of Mathematics, National Taiwan University, Taipei, 106, Taiwan}
\email{jkchen@ntu.edu.tw}

\address{\rm School of Mathematical Sciences, Fudan University,
Shanghai 200433, People's Republic of China}
\email{mchen@fudan.edu.cn}

\address{\rm Shanghai Center for Mathematical Sciences, Fudan University, Jiangwan Campus, 2005 Songhu Road, Shanghai, 200438, People's Republic of China}
\email{chenjiang@fudan.edu.cn}

\thanks{The first author was partially supported by Ministry of Science and Technology and
National Center for Theoretical Sciences of Taiwan. The second author was
supported by National Natural Science Foundation of China (\#11571076, \#11731004) and Program of Shanghai Subject Chief Scientist (\#16XD1400400). The third author was supported by JSPS KAKENHI Grant Number JP16K17558 and World Premier International Research Center Initiative (WPI), MEXT, Japan.}

\begin{document}
\numberwithin{equation}{section}
\begin{abstract} We establish the Noether inequality for projective $3$-folds. More precisely, we prove that the inequality
$$\vol(X)\geq \tfrac{4}{3}p_g(X)-{\tfrac{10}{3}}$$ holds for all projective $3$-folds $X$ of general type with either $p_g(X)\leq 4$ or $p_g(X)\geq 21$, where $p_g(X)$ is the geometric genus and $\vol(X)$ is the canonical volume. 
This inequality is optimal due to known examples found by M. Kobayashi in 1992.
\end{abstract}

\maketitle

\pagestyle{myheadings} \markboth{\hfill J. A. Chen, M. Chen, and C. Jiang
\hfill}{\hfill The Noether inequality for algebraic $3$-folds\hfill}

\section{\bf Introduction}
One main goal of algebraic geometry is to classify algebraic varieties. In birational geometry, a fundamental task is to disclose the exact relation among birational invariants of a given variety. Such kind of strategy is usually referred to as ``algebro-geometric geography'', which may have broader meaning.

We are interested in the relation between two essential birational invariants: the geometric genus and the canonical volume.
For a projective variety $Y$ of dimension $n$, the {\it geometric genus} $p_g(Y)$ is defined by $$p_g(Y):=h^0(Y', K_{Y'})$$ and the {\it canonical volume} $\vol(Y)$ is defined by
$$\vol(Y):=\lim_{m\to \infty}\frac{h^0(Y', mK_{Y'})}{m^n/n!}$$
where $Y'$ is a resolution of $Y$ and $K_{Y'}$ is the canonical divisor. These definitions do not depend on the choice of resolutions because for each integer $m\geq 0$, $h^0(X, mK)$ is a birational invariant in the category of smooth projective varieties (or more generally, the category of normal projective varieties with at worst canonical singularities).
It is also known that $\vol(Y)=K^n_Y$ when $Y$ is a normal projective variety of dimension $n$ with at worst canonical singularities and with nef $K_Y$.
A projective variety is {\it of general type} if it has positive canonical volume.

In this paper, we investigate the so-called ``Noether inequality'' for projective varieties of general type, which describes the lower bound of the canonical volume in terms of the geometric genus. For example, for a complete algebraic curve $C$ of general type, one has $$\vol(C)=2p_g(C)-2$$ simply by Riemann--Roch formula.
In dimension $2$, M. Noether \cite{Noether} proved in 1875 that, for any minimal projective surface $S$ of general type,
$$K_S^2\geq 2p_g(S)-4,$$
or equivalently, for any projective surface $T$, $$\vol(T)\geq 2p_g(T)-4,$$ which is known as the Noether inequality. As one knows, the Noether inequality is a milestone in the history of the surface theory.

Motivated by the study of explicit birational geometry of $3$-folds, one naturally asks for the 3-dimensional analogue of the Noether inequality.
For a projective $3$-fold $X$, to consider the relation between $\vol(X)$ and $p_g(X)$, we may always replace $X$ with one of its minimal models by virtue of 3-dimensional minimal model program (see, for instance, \cite{KMM} and \cite{KM}). Namely we may assume that $X$ is a {\it minimal projective $3$-fold}, i.e., $X$ is a normal projective $3$-fold with at worst $\bQ$-factorial terminal singularities and with nef $K_X$, under which one has $\vol(X)=K_X^3$. Let us briefly recall the history of the Noether inequality problem for $3$-folds:
\begin{enumerate}

\item In 1992, Kobayashi \cite{Kob92} constructed a series of smooth canonically polarized $3$-folds $X$ satisfying $K_X^3=\frac{4}{3}p_g(X)-\frac{10}{3}$. Later in 2017, Chen and Hu \cite{CH17} generalized Kobayshi's method and constructed more series of examples of smooth canonically polarized $3$-folds satisfying the same equality.

\item In 2004, the second author \cite{MChen04MRL} proved the inequality $K_X^3\geq \frac{4}{3}p_g(X)-\frac{10}{3}$ for smooth canonically polarized $3$-folds.

\item In 2006, Catanese, Zhang, and the second author \cite{CCZ06} proved the same inequality for smooth minimal projective $3$-folds of general type.

\item In 2015, the first two authors \cite{CC15} proved the same inequality for Gorenstein minimal projective $3$-folds of general type.

\item The second author also proved that the same inequality holds when $p_g(X)\leq 4$ \cite[Theorem 1.5]{MChen07}.
\end{enumerate}

Unfortunately the known methods are less effective in treating non-Gorenstein 3-folds. In fact, searching for the Noether inequality for arbitrary 3-folds of general type has been a considerably challenging open question.
The aim of this paper is to establish the optimal Noether inequality for ``almost all" projective $3$-folds (see Remark \ref{rem 1}).

Here is our main theorem:
\begin{thm}\label{main} Let $X$ be a projective $3$-fold of general type and either $p_g(X)\leq 4$ or $p_g(X)\geq 21$. Then the inequality
\begin{equation}\vol(X)\geq \frac{4}{3}p_g(X)-{\frac{10}{3}}\label{N ineq}\end{equation}
holds.
\end{thm}
 The inequality is {optimal} due to examples constructed by Kobayashi \cite{Kob92} and Chen--Hu \cite{CH17}.

In order to show Theorem \ref{main}, firstly we may assume that $X$ is minimal.  We may always assume that $p_g(X)\geq 3$ since it is clear otherwise. Then we can consider the map $\varphi_1=\Phi_{|K_X|}$ defined by the canonical linear system $|K_X|$ and discuss on the {\it canonical dimension} $d_X:=\dim \overline{\varphi_1(X)}$. Recall that a {\it $(1, 2)$-surface} is a smooth projective surface $S$ of general type with $\vol(S)=1$ and $p_g(S)=2$. In fact, we prove the following theorem:

\begin{thm}\label{main2} Let $X$ be a minimal projective $3$-fold of general type. Assume that one of the following holds:
\begin{enumerate}
\item $d_X\geq 2$; or
\item $d_X=1$ and $|K_X|$ is not composed with a rational pencil of $(1,2)$-surfaces; or
\item $d_X=1$, $|K_X|$ is composed with a rational pencil of $(1,2)$-surfaces, and either $p_g(X)\leq 4$ or $p_g(X)\geq 21$.
\end{enumerate}
Then Inequality \eqref{N ineq} holds.
\end{thm}

\begin{remark}\label{rem 1}
 By Theorem \ref{main2}, if Inequality \eqref{N ineq} does not hold for a minimal projective $3$-fold $X$ of general type, then $5\leq p_g(X)\leq 20$, $K_X^3<70/3$, and $|K_X|$ is composed with a rational pencil of $(1,2)$-surfaces. We hope to study such exceptional cases in our next work. Also note that there are only finitely many families of such $3$-folds since minimal projective $3$-folds of general type with $K^3< c$ form a bounded family for any $c>0$ by \cite[Theorem 4]{MST}. In other words, Theorem \ref{main2} proves that the optimal Noether inequality holds for all but finitely many families of minimal projective $3$-folds (up to deformation).
\end{remark}

We briefly explain the difficulty of this problem and our strategy. Let $X$ be a minimal projective $3$-fold. The rough idea is to find a resolution $\pi: W\to X$ and a divisor $S$ on $W$ such that $\pi^*K_X\geq S$, then we can use $K_X^3\geq (\pi^*K_X|_S)^2$ to estimate the lower bound of $K_X^3$. One difficulty here is that both intersection numbers are no longer integers (which is not the case for previous works) and the singularities of $X$ make the situation more complicated.

In order to estimate $(\pi^*K_X|_S)^2$, we manage to find a comparison theorem between $\pi^*K_X|_S$ and $\sigma^*(K_{S_0})$, where $\sigma: S\to S_0$ is the minimal model of $S$.
Such a comparison theorem was known in previous works, but it heavily depends on a special choice of the resolution $\pi$ and contains tedious computations on exceptional divisors. In this paper, we establish the similar comparison theorem, independent of the resolution, as a simple application of an extension theorem (see Subsection \ref{sec ext}).

Hence the problem is reduced to estimate $K_{S_0}^2$. Most cases can be done by this way directly, except for $3$ main difficult cases: (1) $|K_X|$ gives a fibration of curves of genus $2$; (2) $|K_X|$ is composed with a pencil of $(1,1)$-surfaces; (3) $|K_X|$ is composed with a pencil of $(1,2)$-surfaces.
Here recall that for two integers $m>0$ and $n\geq 0$, an {\it $(m, n)$-surface} is a smooth projective surface $S$ of general type with $\vol(S)=m$ and $p_g(S)=n$.

In Case (1), the surface $S_0$, we find, is fibered by curves of genus $2$, and we treat this case by establishing a nice inequality for $K_{S_0}^2$ (See Subsection \ref{sec g=2}) and this method gives a new simplified proof compared to that for the Gorenstein case.

In Case (2), somehow we can use the classical method in \cite{MChen04JMSJ} to handle it.

Case (3) is the most difficult one. Recall that the case $p_g(X)\leq 4$ was treated in \cite[Theorem 1.5]{MChen07}.
In order to find a good $S$, we need to assume that $\Mov|K_X|$ is base point free. Of course this is not always true. However, we show that this is the case when $p_g(X)\geq 21$ by establishing an inequality on the pencil of surfaces which guarantees that if $\Mov|K_X|$ is not free, then $p_g(X)$ is bounded from above in terms of the global log canonical threshold of the surface (see Section \ref{section pencil ineq}). The number $21$ here comes from the estimate of the global log canonical thresholds of minimal $(1,2)$-surfaces, which is given in Appendix by Koll\'{a}r (see Theorem \ref{12alpha new})\footnote{Our original estimation on glct of $(1,2)$-surfaces was built on a case-by-case study of $g=2$ fibrations. We are grateful to Koll\'ar for providing an elegant and short proof in the Appendix.}.
Provided $\Mov|K_X|$ is base point free, we can firstly prove an inequality between the canonical volume and the geometric genus which is slightly weaker than the Noether inequality we expect. But during the proof, we can get some geometric information about the exceptional cases where we get weaker inequality.
Together with a very detailed investigation of the geometry of such exceptional cases, we manage to prove the expected inequality by using weak positivity of certain direct image sheaves. This method is new even for the Gorenstein case (in \cite{CCZ06}, those exceptional cases for Gorenstein minimal projective $3$-folds are treated by a totally different method, which is not applicable for non-Gorenstein minimal projective $3$-folds).

Therefore, even if $X$ is assumed to be Gorenstein, then this paper gives new proofs for the Noether inequality \eqref{N ineq} in the situations of \cite[Theorem 4.1]{CCZ06} and of \cite[Theorem 3.1]{CC15} (see Theorems \ref{last} and \ref{d2g2}).

At the end of  the introduction, we raise some open questions with remarks which are related to the Noether inequality problem.
\begin{question} Does Inequality \eqref{N ineq} hold for projective $3$-folds $X$ with $5\leq p_g(X)\leq 20$?
\end{question}

\begin{question} Is there a classification for all (minimal) projective $3$-folds $X$ satisfying $\vol(X)= \frac{4}{3}p_g(X)-{\frac{10}{3}}$? Or is there a new method to construct more examples, especially non-Gorenstein examples, which satisfy the Noether equality?
\end{question}
Note that all the known examples appear in Kobayashi \cite{Kob92} and Chen--Hu \cite{CH17} are Gorenstein minimal.

\begin{question}Is there a ``second Noether inequality" for projective $3$-folds? Namely, is there a real number $b<\frac{10}{3}$, such that for a projective $3$-folds $X$, if $\vol(X)> \frac{4}{3}p_g(X)-{\frac{10}{3}}$, then $\vol(X)\geq \frac{4}{3}p_g(X)-b$?
\end{question}
For a projective surface $S$, it is clear that, if $\vol(S)>2p_g(S)-4$, then $\vol(S)\geq 2p_g(S)-3$. But it becomes complicated for $3$-folds since $\vol(X)$ is no longer an integer.

\begin{question} How about the Noether inequality in higher dimensions?
\end{question}
Note that the existence of the Noether type inequality in higher dimension was proved in \cite[Corollary 5.1]{CJZ}, but so far there is no concrete formula. For some interesting examples, one may refer to \cite{CL} for a recent development.

\begin{question} How about the Noether inequality in positive characteristics?
\end{question}
The Noether inequality for surfaces in arbitrary characteristic was established by Liedtke \cite{Lie08}. But for the Noether type inequality for $3$-folds in positive characteristics, we know almost nothing.




\section{\bf Preliminaries}
\subsection{Notation and conventions}\
Throughout this paper, we work over any algebraically closed field of characteristic $0$.

 We adopt the standard notation and definitions in \cite{KMM} and \cite{KM}, and will freely use them.

A {\it log pair} $(X, B)$ consists of a normal projective variety $X$ and an effective $\mathbb{Q}$-divisor $B$ on $X$ such that $K_X+B$ is $\mathbb{Q}$-Cartier.

Let $f: Y\rightarrow X$ be a log
resolution of the log pair $(X, B)$, write
$$
K_Y =f^*(K_X+B)+\sum a_iF_i,
$$
where $\{F_i\}$ are distinct prime divisors. The log pair $(X,B)$ is called
\begin{enumerate}[label=(\alph*)]
\item \emph{kawamata log terminal} (\emph{klt},
for short) if $a_i> -1$ for all $i$;

\item \emph{log canonical} (\emph{lc}, for
short) if $a_i\geq -1$ for all $i$;

\item \emph{terminal} if $a_i> 0$ for all $f$-exceptional divisors $F_i$ and all $f$;
\item \emph{canonical} if $a_i\geq 0$ for all $f$-exceptional divisors $F_i$ and all $f$.
\end{enumerate}
Usually we write $X$ instead of $(X,0)$ in the case $B=0$.

For two integers $m>0$ and $n\geq 0$, an {\it $(m, n)$-surface} is a smooth projective surface $S$ of general type with $\vol(S)=m$ and $p_g(S)=n$.

A $\bQ$-divisor is said to be {\it $\bQ$-effective} if it is $\bQ$-linear equivalent to an effective $\bQ$-divisor.

For a $\bQ$-divisor $D$ and a prime divisor $E$, we denote $\coeff_{E}(D)$ by the coefficient of $E$ appearing in the expression of $D$.

For two linear systems $|A|$ and $|B|$, we write $|A|\preceq |B|$ if there exists an effective divisor $F$ such that $$|B|\supseteq |A|+F.$$ In particular, if $A\leq B$ as divisors, then $|A|\preceq |B|$.

\subsection{Rational maps defined by linear systems of Weil divisors}\label{b setting}\

Let $X$ be a normal projective $3$-fold.
Consider an effective $\bQ$-Cartier Weil divisor $D$ on $X$ with $h^0(X, D)\geq 2$. We study the rational map defined by $|D|$, say
$$X\overset{\Phi_{|D|}}{\dashrightarrow} \bP^{h^0(D)-1}$$ which is
not necessarily well-defined everywhere. By Hironaka's big
theorem, we can take successive blow-ups $\pi: W\rightarrow X$ such
that:
\begin{itemize}
\item [(i)] $W$ is smooth projective;
\item [(ii)] the movable part $|M|$ of the linear system
$|\rounddown{\pi^*(D)}|$ is base point free and, consequently,
the rational map $\gamma=\Phi_{|D|}\circ \pi$ is a projective morphism.
\end{itemize}
Let $W\overset{f}\longrightarrow \Gamma\overset{s}\longrightarrow Z$
be the Stein factorization of $\gamma$ with $Z=\gamma(W)\subseteq
\bP^{h^0(D)-1}$. We have the following commutative
diagram:
$$\xymatrix@=4.5em{
W\ar[d]_\pi \ar[dr]^{\gamma} \ar[r]^f& \Gamma\ar[d]^s\\
X \ar@{-->}[r]^{\Phi_{|D|}} & Z}
$$

There are $3$ cases according to $\dim \Gamma$.
\begin{enumerate}
\item If $\dim(\Gamma)=3$, then a general
member of $|M|$ is a smooth projective surface by
Bertini's theorem, and ${|D|}$ defines a generically finite map.

\item If $\dim(\Gamma)=2$, then a general
member $S$ of $|M|$ is a smooth projective surface by
Bertini's theorem, and a general fiber $C$ of $f$ is a smooth curve of genus $g\geq 2$. In this case, we say that $|D|$ {\it gives a fibration of curves of genus $g$}.

\item If $\dim(\Gamma)=1$, then $\Gamma$ is a smooth curve and
a general fiber $F$ of $f$ is a smooth projective surface
by Bertini's theorem. In this case, we say that
$|D|$ {\it is composed with a pencil of surfaces}. Moreover, if $F$ is an $(m,n)$-surface for some integers $m$ and $n$, then we say that
$|D|$ {\it is composed with a pencil of $(m,n)$-surfaces}. This pencil is said to be {\it rational} if $\Gamma\simeq \bP^1$. We may write
$M=\sum_{i=1}^a F_i$ where $F_i$ is a smooth fiber of $f$ for each $i$. It is easy to see that $a= h^0(D)-1$ if $\Gamma\simeq \bP^1$, and $a\geq h^0(D)$ if $\Gamma\not \simeq \bP^1$. \end{enumerate}

\begin{lem}\label{S^2=aC}
Keep the above setting. In Case (2), we can write $S|_S\equiv aC$ for some integer $a\geq h^0(D)-2$, where $C$ can be viewed as a general fiber of the restricted fibration $f|_{S}: S\to f(S).$ Moreover, if the equality holds, then $s: \Gamma\to Z$ is an isomorphism and $Z\subseteq \bP^{h^0(D)-1}$ is a surface of minimal degree (\cite[Exercise IV.18.4]{Beauville}).
\end{lem}
\begin{proof}
The inequality follows by the fact that $a=\deg s\cdot \deg Z$ and $Z\subseteq \bP^{h^0(D)-1}$ is non-degenerate. If the equality holds, then $s$ is birational and $Z\subseteq \bP^{h^0(D)-1}$ is a surface of minimal degree. In this case, $Z$ is normal and hence $s$ is an isomorphism.
\end{proof}

\subsection{A restriction comparison by virtue of extension theorem}\label{sec ext}\

We will use the following special form of an extension theorem due to Kawamata.
\begin{thm}[{\cite[Theorem A]{Kaw99}}]\label{kaw ext}
Let $V$ be a smooth variety and $D$ a smooth divisor on $V$. Assume that $K_V+D\sim_\bQ A+B$ where $A$ is an ample $\bQ$-divisor and $B$ is an effective $\bQ$-divisor such that $D\not\subseteq \Supp(B)$. Then
the
natural homomorphism $$H^0(V, \OO_V(m(K_V + D))) \to H^0(D, \OO_D(mK_D))$$ is surjective for all $m\geq 2$.
\end{thm}

By applying the above extension theorem, we get the following very useful corollary, which can be viewed as a generalization of \cite[Lemma 3.4]{CZuo08} or \cite[Lemma 3.7]{CZ08}, but the proof here is much more simple. The idea of the proof is from \cite[Subsection 2.4]{CZ16}.

\begin{cor}\label{KXKS}
Let $X$ be a minimal projective $3$-fold of general type and $D$ a semi-ample Weil divisor on $X$. Let $\pi:W\to X$ be a resolution and $S$, a semi-ample divisor on $W$, is assumed to be a smooth surface of general type. Assume that $\lambda\pi^*D- S$ is $\bQ$-effective for some positive rational number $\lambda$. Then $\pi^*(K_X+\lambda D)|_S-\sigma^*K_{S_0}$ is $\bQ$-effective on $S$, where $\sigma:S\to S_0$ is the contraction onto the minimal model $S_0$.
\end{cor}
\begin{proof}
Since $X$ is of general type, $K_W$ is big. Since $S$ is semi-ample, by Theorem \ref{kaw ext},
$$H^0(W, \OO_W(m(K_W + S))) \to H^0(S, \OO_S(mK_S))$$ is surjective for all $m\geq 2$.
Denote by $\hat{M}_m$ the movable part of $|m(K_W+S)|$. Note that for all $m\geq 4$, the movable part of $|mK_S|$ is just $|m\sigma^*K_{S_0}|$  as $|mK_{S_0}|$ is base point free (\cite{Bom73}).
Now take a sufficiently divisible $m$ such that $mK_X$ and $m\lambda D$ are Cartier and base point free, and $|m\lambda\pi^*D|\succeq |mS|$. In particular, $$|m(K_W+\pi^*(\lambda D))|\succeq |m(K_W + S)|.$$
Note that $|m\pi^*(K_X+\lambda D)|=\Mov|m(K_W+\pi^*(\lambda D))|$.
Hence by \cite[Lemma 2.7]{MChen01}, we have
\begin{align*}
|m\pi^*(K_X+\lambda D)||_S\succeq \hat{M}_m|_S\succeq |m\sigma^*K_{S_0}|,
\end{align*}
which means that $\pi^*(K_X+\lambda D)|_S-\sigma^*K_{S_0}$ is $\bQ$-effective on $S$.
\end{proof}

\subsection{Weak positivity of direct images}\

Recall the definition of weak positivity by Viehweg.
\begin{definition}[\cite{Vie83}] Let $X$ be a smooth projective variety and $\FF$ a torsion-free coherent sheaf on $X$. We say that $\FF$ is weakly positive on $X$ if there exists some Zariski open subvariety $U \subseteq X$ such that for every ample invertible sheaf $\HH$ and every positive integer $\alpha$, there exists some positive integer $\beta$ such that $({S}^{\alpha\beta}\FF)^{**}\otimes \HH^\beta$ is generated by global sections over $U$, which means that the natural map
$$
H^0(X,({S}^{\alpha\beta}\FF)^{**} \otimes \HH^\beta )\otimes \OO_X \to ({S}^{\alpha\beta}\FF)^{**} \otimes \HH^\beta
$$
is surjective over $U$.
Here $({S}^{k}\FF)^{**}$ denotes the reflexive hull of the symmetric product ${S}^{k}\FF$.
\end{definition}

In the very last part of the proof (see Claim \ref{claim K-F pseff}), we need to estimate intersection numbers coming from the relative canonical sheaf, where weak positivity of direct images is involved in an effective way.
We need the following result that was originally developed by Viehweg \cite{Vie83} and generalized by Campana \cite{Cam04} and Lu \cite{Lu02}. Here we only state a simplified version of {\cite[Theorem 4.13]{Cam04}}.
\begin{thm}[{\cite[Theorem 4.13]{Cam04}}]\label{thm wp} Let $g: Y\to Z$ be a surjective morphism between smooth projective varieties and $D$ a reduced divisor on $Y$ with simple normal crossing support. Then $g_*\OO_Y(m(K_{Y/Z} +D))$ is torsion free and weakly positive for every integer $m>0$.
\end{thm}

\subsection{Global log canonical thresholds}\

We recall the definition of log canonical thresholds and introduce the concept of global log canonical thresholds for minimal projective varieties of general type.

Let $(X, B)$ be a log canonical pair and $D\geq 0$ be a $\bQ$-Cartier $\bQ$-divisor. The
{\it log canonical threshold} of $D$ with respect to $(X, B)$ is defined by
$$\lct(X, B; D) = \sup\{t\geq 0 \mid (X, B+ tD) \text{ is lc}\}.$$

\begin{definition}
Let $Y$ be a normal projective variety with at worst klt singularities such that $K_Y$ is nef and big. We define the {\it global log canonical threshold} ({\it \glct}, for short) of $Y$ as the following:
\begin{align*}
\glct(Y){}&=\inf\{\lct(Y; D)\mid 0\leq D\sim_\bQ K_Y\}\\
{}&=\sup\{t\geq 0\mid (Y, tD) \text{ is lc for all }0\leq D\sim_\bQ K_Y\}.
\end{align*}
\end{definition}

In general, {\glct} is very difficult to compute. In this paper, we are mainly interested in the lower bound of $\glct(S)$ for a minimal $(1,2)$-surface $S$.
Note that in an earlier version of this paper, we showed that the glct of minimal $(1,2)$-surfaces are at least $\frac{1}{13}$, which depends on detailed analysis of Ogg's list of genus two fibrations. Later J\'{a}nos Koll\'{a}r sent us his note which gives a delicate and short proof that the glct of minimal $(1,2)$-surfaces are at least $\frac{1}{10}$ (see Theorem \ref{lcth.from.h0.cor}) and kindly allowed us to include his note in this paper (see Appendix A). We rephrase his result here.
\begin{thm}[Koll\'{a}r]\label{12alpha new}
Let $S$ be a minimal $(1,2)$-surface. Then $$\glct(S)\geq \frac{1}{10}.$$
\end{thm}
\begin{proof}
Fix an effective $\bQ$-divisor $B\sim_\bQ K_S$, it suffices to show that $(S, \frac{1}{10}B)$ is lc.
Denote by $\bar{S}$ the canonical model of $S$ and $\tau: S \to \bar{S}$ the induced map. Consider the effective $\bQ$-divisor $\tau_*B\sim_\bQ K_{\bar{S}}$. Then Theorem \ref{lcth.from.h0.cor} shows that $(\bar{S}, \frac{1}{10}\tau_*B)$ is lc. Since $\tau$ is crepant, $(S, \frac{1}{10}B)$ is also lc.
\end{proof}

\begin{remark}
The concept of {\glct} we defined here is an analogue of the global log canonical thresholds (also called alpha-invariants) of Fano varieties (\cite{Tia87, Dem08}). Unlike the {\glct} of Fano varieties, it is surprising that we could not find any related study of this invariant in literature for minimal projective varieties of general type. The reason might be that the {\glct} of Fano varieties have been found to possess many important applications (e.g. on the existence of K\"ahler--Einstein metrics), whereas the {\glct} of minimal projective varieties of general type is not the case in practice.
\end{remark}

However, in Section \ref{section pencil ineq}, we establish an interesting inequality for minimal projective $3$-folds which admit a pencil of surfaces, where the {\glct} of minimal surfaces are involved very naturally. We hope to find more interesting applications for the {\glct} of minimal projective varieties of general type in the future.

\subsection{An inequality for surfaces admitting a genus $2$ fibration}\label{sec g=2}\

\begin{prop}\label{vol S>8/3}
Let $S$ be a smooth projective surface of general type and $T$ a smooth complete curve.
Suppose that $f: S\to T$ is a fibration of which the general fiber $C$ is of genus $2$ . Assume that $p_g(S)\geq 3$ and $K_S\equiv nC+G$ for some effective integral divisor $G$ on $S$ and a positive integer $n$. Then $$\vol(S)\geq \frac{8}{3}(n-1).
$$
\end{prop}
\begin{proof}
Denote by $S_0$ the minimal model of $S$ and $\sigma:S \to S_0$ the contraction map. Since $p_g(S)\geq 3$, $(\sigma^*K_{S_0}\cdot C)\geq 2$ by the Hodge index theorem (see, for example, \cite[Lemma 2.4]{exp3}). Since $p_g(C)=2$, $(K_S\cdot C)=2$ and hence $(\sigma^*K_{S_0}\cdot C)= 2$. In particular, this means that all $\sigma$-exceptional divisors are contracted by $f$, and hence the induced map $S_0\to T$ is a morphism. All conditions are the same after replacing $S$ with $S_0$. Hence we may and do assume that $S$ is minimal from now on and so $\vol(S)=K^2_S$.

Write $G=\Gamma+V$, where $\Gamma$ is the horizontal part and $V$ is the vertical part with respect to $f$. Then $$(\Gamma\cdot C)=(K_S\cdot C)=2.$$
As $\Gamma$ is integral, there are $3$ cases:
\begin{enumerate}
\item $\Gamma$ is a prime divisor; or
\item $\Gamma=\Gamma_1+\Gamma_2$ where $\Gamma_1$ and $\Gamma_2$ are distinct prime divisors; or
\item $\Gamma=2\Gamma_1$ where $\Gamma_1$ is a prime divisor.
\end{enumerate}

In Case (1), note that
$$
((K_S+\Gamma)\cdot \Gamma)\geq -2
$$
and $K_S$ is nef, we have
\begin{align*}
2K_S^2{}&=(2K_S\cdot (nC+\Gamma+V))\\
{}&\geq 4n+(2K_S\cdot \Gamma)\\
{}&=4n+((K_S+nC+\Gamma+V)\cdot \Gamma)\\
{}&\geq 4n+((K_S+\Gamma)\cdot \Gamma)+2n\\
{}&\geq 6n-2.
\end{align*}

In Case (2), note that for $i=1,2$,
$$
((K_S+\Gamma_i)\cdot \Gamma_i)\geq -2.
$$
This implies that
$$
((K_S+\Gamma)\cdot \Gamma)\geq ((K_S+\Gamma_1)\cdot \Gamma_1) +((K_S+\Gamma_2)\cdot \Gamma_2)\geq -4.
$$
Arguing as Case (1), it is easy to see that
\begin{align*}
2K_S^2{}&\geq 6n-4.
\end{align*}

In Case (3), note that
$$
((K_S+\Gamma_1)\cdot \Gamma_1)\geq -2,
$$ we have
\begin{align*}
3K_S^2{}&=(3K_S\cdot (nC+\Gamma+V))\\
{}&\geq 6n+(3K_S\cdot \Gamma)\\
{}&=6n+((2K_S+nC+2\Gamma_1+V)\cdot 2\Gamma_1)\\
{}&\geq 6n+((2K_S+2\Gamma_1)\cdot 2\Gamma_1)+2n\\
{}&\geq 8n-8.
\end{align*}

Summarizing all cases, we proved the inequality.
\end{proof}



\section{\bf Some properties of a pencil of surfaces over a curve}\label{section pencil ineq}

{}First we establish a geometric inequality, on a pencil of surfaces over a curve, which is inspired by the idea in proving \cite[Lemma 4.2]{CJ17}. Historically, similar ideas using the connectedness of the locus of log canonical singularities to bound singularities and multiplicities have been used in the study of Fano varieties (see \cite{Chel99} for example), so it is interesting to see that in our situation, we may also apply this idea to the study of varieties of general type.

\begin{prop}\label{pencil bound}
Let $X$ be a minimal projective $3$-fold of general type. Assume that there exists a resolution $\pi:W\to X$ such that $W$ admits a fibration structure $f: W\to \Gamma$ onto a smooth curve $\Gamma$. Denote by $F$ a general fiber of $f$ and $F_0$ the minimal model of $F$. Assume that
\begin{enumerate}
 \item there exists a $\pi$-exceptional prime divisor $E_0$ on $W$ such that $(\pi^*(K_X)|_F\cdot E_0|_F)>0$, and
 \item $\pi^*(K_X)\sim_\bQ bF+D$ for some rational number $b>0$ and an effective $\bQ$-divisor $D$ on $W$.
 \end{enumerate}
 Then $b\leq \frac{2}{\glct(F_0)}.$ Moreover, this inequality is strict if $b>2K^2_{F_0}$
\end{prop}
\begin{proof}
Note that, according to the projection formula, the assumptions in the proposition still hold if we replace $W$ with any higher birational model over $W$ (that is, a smooth variety $W'$ with a proper birational morphism $W'\to W$) and replace $E_0$ with its proper transform.
Take $g:W_0\to \Gamma$ to be a relative minimal model of $f: W\to \Gamma$ (for the definition, see \cite[Definition 3.50]{KM}), of which the general fiber is $F_0$. Modulo a further birational modification, we may assume that $f$ factors through $g$ by a morphism $\zeta: W\to W_0$.
We may write
$$K_W=\pi^*K_X+E_\pi,$$
where $E_\pi$ is an effective $\pi$-exceptional $\bQ$-divisor. Being a minimal model of $W$, $X$ is also a minimal model of $W_0$ and we may write
$$
\zeta^*(K_{W_0})=\pi^*(K_X)+\hat{E},
$$
where $\hat{E}$ is an effective $\pi$-exceptional $\bQ$-divisor.

Take a general fiber $F$ of $f$, by the assumption,
there exists a  $\pi$-exceptional prime divisor $E_0$ on $W$ such that $(\pi^*(K_X)|_F\cdot E_0|_F)>0$.
In particular, $E_0|_F$ is not contracted by $\pi.$ Hence there exists a curve $\Gamma_X\subseteq X$ such that $(K_X\cdot \Gamma_X)>0$ and that
$$\Gamma_X\subseteq \pi(E_0\cap F)\subseteq \pi(E_0).$$ On the other hand, since  $\pi(E_0)$ is a subvariety of codimension at least $2$, we see $\Gamma_X= \pi(E_0)$. In particular, $\Gamma_X$ is independent of $F$, and for any general fiber $F$ of $f$, $\Gamma_X= \pi(E_0\cap F)$.

By the assumption, we have
$$
\pi^*(K_X)\sim_\bQ bF+D.
$$
Take $w=2/b$. Pick two general fibers $F_1$ and $F_2$ of $f$ and consider the pair
$$
(W, -E_\pi+wD+F_1+F_2),
$$
which can be assumed to have simple normal crossing support modulo a further birational modification.
Note that
$$ -(K_W-E_\pi+wD+F_1+F_2)\sim_\bQ -(1+w)\pi^*(K_X)$$
is $\pi$-nef and ${\pi}_*(-E_\pi+wD+F_1+F_2)\geq 0$ since
$E_\pi$ is $\pi$-exceptional.
Denote by $G$ the support of the effective part of $\rounddown{-E_\pi+wD}+F_1+F_2$.
By the Connectedness Lemma (see \cite[Theorem 5.48]{KM}),
$$
G\cap \pi^{-1}(x)
$$
is connected for any point $x\in X$.

We claim that there exists a prime divisor $E_1$ on $W$ such that $$\coeff_{E_1}(-E_\pi+wD)\geq 1$$ and that $\pi(E_1\cap F)$ contains $\Gamma_X$ for a general fiber $F$ of $f$.
Consider a point $x\in \Gamma_X\subseteq X$, then $F\cap \pi^{-1}(x) \neq \emptyset$ for a general $F$. In particular, $F_1\cap \pi^{-1}(x)\neq \emptyset$ and $F_2\cap \pi^{-1}(x)\neq \emptyset$, which are two disconnected sub-sets in $G\cap \pi^{-1}(x)$. Since $G\cap \pi^{-1}(x)$ is connected, there exists a curve $B_x\subseteq G\cap \pi^{-1}(x)$ such that $B_x\cap F_1\cap \pi^{-1}(x)\neq \emptyset$ and $B_x\not\subseteq F_1\cap \pi^{-1}(x)$. Moving $x$ in $\Gamma_X$, we get an infinite set of curves $\{B_x\}$, which means that
there exists a prime divisor $E_1\subseteq G$ such that $B_x\subseteq E_1\cap \pi^{-1}(x)$ for infinitely many $x\in \Gamma_X$. Hence
$$x\in \pi( B_x\cap F_1\cap \pi^{-1}(x))\subseteq \pi(E_1\cap F_1\cap \pi^{-1}(x))\subseteq \pi(E_1\cap F_1)$$
for infinitely many $x\in \Gamma_X$. This implies that $\Gamma_X\subseteq \pi(E_1\cap F_1).$
By the construction of $E_1$, $E_1\subseteq G$ and it is clear that $E_1$ is different from $F_1$ and $F_2$, hence $$\coeff_{E_1}(-E_\pi+wD)\geq 1.$$
By the generality of $F_1$, $\Gamma_X\subseteq \pi(E_1\cap F)$ for a general $F$.

Now denote
$$
\Delta_W:=-E_\pi+wD+F_1+F_2+(1+w)\hat{E},
$$
then
$$ K_W+\Delta_W\sim_\bQ (1+w)\pi^*(K_X)+(1+w)\hat{E}\sim_\bQ(1+w)\zeta^*(K_{W_0}).$$
We may write
$$ K_W+\Delta_W=\zeta^*(K_{W_0}+\Delta_{W_0}),$$
where $\Delta_{W_0}=\zeta_*(\Delta_{W})\sim_\bQ wK_{W_0}$. Also note that $$\Delta_{W_0}=\zeta_*(\Delta_{W})=\zeta_*(\Delta_{W}+E_\pi-\hat{E})\geq 0$$ since $E_\pi-\hat{E}=K_W-\zeta^*(K_{W_0})$ is $\zeta$-exceptional.
Restricting on a general fiber $F$ of $f$, we have
$$ K_F+\Delta_W|_F=\zeta_F^*(K_{F_0}+\Delta_{W_0}|_{F_0}),$$
where $\zeta_F=\zeta|_F:F\to F_0$.
By the construction,
\begin{align}
\coeff_{E_1}\Delta_W\geq 1+(1+w)\coeff_{E_1}\hat{E}\geq 1\label{Delta>1}
\end{align}
and $E_1\cap F\neq \emptyset$, hence $\Delta_W|_F$ contains a component with coefficient $\geq 1$.
This implies that $(F_0, \Delta_{W_0}|_{F_0})$ is not klt. On the other hand, $\Delta_{W_0}|_{F_0}\sim_\bQ wK_{F_0}$,
hence
\begin{equation}
\frac{2}{b}=w\geq \lct\Big(F_0;\frac{1}{w}\Delta_{W_0}|_{F_0}\Big)\geq \glct(F_0). \label{lcttt}
\end{equation}
Moreover, if $E_1\subseteq \Supp(\hat{E}),$ then $\coeff_{E_1}\Delta_W>1$ by Inequality \eqref{Delta>1}, and Inequality \eqref{lcttt}  becomes a strict one. 

To finish the proof, we only need to show that, whenever $b>2K_{F_0}^2$, then $E_1\subseteq \Supp(\hat{E}).$ Assume, to the contrary, that $E_1\not\subseteq \Supp(\hat{E})$. Then for a general fiber $F$ of $f$, $E_1|_F$ has no common component with $\Supp(\hat{E}|_F)$. Since $\Gamma_X\subseteq \pi(E_1\cap F)$, we can find a curve $\Gamma_W\subseteq E_1\cap F$ and $\Gamma_X=\pi(\Gamma_W)$ such that $\Gamma_W\not\subseteq \Supp(\hat{E}|_F)$. Recall that $(K_X\cdot \Gamma_X)>0$.
Hence
\begin{align*}
(\zeta_F^*(K_{F_0})\cdot \Gamma_W){}&=(\zeta^*(K_{W_0})|_F\cdot \Gamma_W)\\
{}&=((\pi^*(K_{X})|_F+\hat{E}|_F)\cdot \Gamma_W)\\
{}&\geq (\pi^*(K_{X})|_F\cdot \Gamma_W)\\
{}&=(K_X\cdot \pi_*(\Gamma_W))>0.
\end{align*}
In particular, $\Gamma_W$ is not contracted by $\zeta_F$ and $$(K_{F_0}\cdot {\zeta_F}_*(\Gamma_W))=(\zeta_F^*(K_{F_0})\cdot \Gamma_W)\geq 1$$
since it is an integer.
On the other hand, we know that 
$$\Delta_{W_0}|_{F_0}={\zeta_F}_*(\Delta_{W}|_{F})\geq {\zeta_F}_*(E_1|_{F})\geq {\zeta_F}_*(\Gamma_W).$$
Hence
\begin{align*}
\frac{2}{b}K_{F_0}^2{}&=wK_{F_0}^2=(\Delta_{W_0}|_{F_0}\cdot K_{F_0})\geq ({\zeta_F}_*(\Gamma_W)\cdot K_{F_0})\geq 1,
\end{align*}
which is a contradiction.
\end{proof}

Condition (1) of Proposition \ref{pencil bound} seems to be technical, but as a matter of fact, it has very natural geometric meaning by the following lemma. Namely, the absence of such $E_0$ is equivalent to the minimality of $W_0$.

\begin{lem}\label{lem no E=minimal} Let $X$ be a minimal projective $3$-fold of general type. Assume that there exists a resolution $\pi:W\to X$ such that $W$ admits a fibration $f: W\to \Gamma$ onto a smooth curve $\Gamma$. Take $g: W_0\to \Gamma$ to be a relative minimal model of $f$. We may and do assume that the induced map $\zeta: W\to W_0$ is a morphism. Denote by $F$ a general fiber of $f$ and $F_0$ the minimal model of $F$ with the induced map $\zeta_F=\zeta|_F: F\to F_0$. Then the following statements are equivalent:
\begin{enumerate}
 \item there does not exist any $\pi$-exceptional prime divisor $E_0$ on $W$ such that $(\pi^*(K_X)|_F\cdot E_0|_F)>0$;
 \item $W_0$ is a minimal projective $3$-fold;
 \item $\pi^*(K_X)|_F=\zeta^*(K_{W_0})|_F=\zeta_F^*(K_{F_0})$;
 \item $\pi^*(K_X)=\zeta^*(K_{W_0})$.
 \end{enumerate}
 Moreover, if $p_g(\Gamma)>0$, then the above conditions hold.
 \end{lem}
\begin{proof}
It is easy to see that $(2)$ and $(4)$ are equivalent.
Thus we prove the following implications:
$$
(1)\implies (3)\implies (2)\iff (4)\implies (1).
$$

Since $X$ is a minimal model of $W$, it is also a minimal model of $W_0$, and we may write
$$
\zeta^*(K_{W_0})=\pi^*(K_X)+\hat{E},
$$
where $\hat{E}$ is an effective $\pi$-exceptional $\bQ$-divisor.
Restricting on a general fiber $F$ of $f$, we have
\begin{align}
\zeta_F^*(K_{F_0})=\zeta^*(K_{W_0})|_F=\pi^*(K_X)|_F+\hat{E}|_F.\label{K=K+E}
\end{align}

Suppose that Condition (1) holds, which means that, for any  $\pi$-exceptional prime divisor $E_0$ on $W$, $(\pi^*(K_X)|_F\cdot E_0|_F)=0$ since $\pi^*(K_X)$ and $F$ are nef. In particular, $(\pi^*(K_X)|_F\cdot \hat{E}|_F)=0$.
Since $\zeta^*(K_{W_0})|_F$ is nef and big, $m\zeta^*(K_{W_0})|_F$ is a $1$-connected divisor by the Hodge index theorem for  any integer $m>0$. Hence $\hat{E}|_F=0$ and  $\zeta^*(K_{W_0})|_F=\pi^*(K_X)|_F$, which proves (3).

Suppose that Condition (3) holds.  Then by Equation \eqref{K=K+E}, $\hat{E}|_F=0$ which means that $\hat{E}$ is contracted to points by $f$.
It suffices to show that $K_{W_0}$ is nef. Assume, to the contrary, that $K_{W_0}$ is not nef. Then there exists a curve $\Gamma'$ on $W$ such that $(\zeta^*(K_{W_0})\cdot \Gamma')<0$. Note that $\Gamma'$ is not contracted by $f$ since $\zeta^*(K_{W_0})$ is $f$-nef by definition. In particular, $\Gamma'\not\subseteq \Supp(\hat{E})$ since $\hat{E}$ is contracted to points by $f$.
Then
$$
(\zeta^*(K_{W_0})\cdot \Gamma')=(\pi^*K_X\cdot \Gamma')+(\hat{E}\cdot \Gamma')\geq 0
$$
since $K_X$ is nef, which is a contradiction. Hence $K_{W_0}$ is nef, which proves (2).

Suppose that Condition (2) holds. Since $X$ and $W_0$ are both minimal, we have
$$
\zeta^*(K_{W_0})=\pi^*(K_X)
$$
and $X$ and $W_0$ are isomorphic in codimension one (see \cite[Theorem 3.52]{KM}).
For any  $\pi$-exceptional prime divisor $E_0$ on $W$, it is also $\zeta$-exceptional, and therefore ${E_0}|_F$ is $\zeta_F$-exceptional for a general $F$. Then
$$
(\pi^*(K_X)|_F\cdot E_0|_F)=(\zeta^*(K_{W_0})|_F\cdot E_0|_F)=(\zeta_F^*(K_{F_0})\cdot E_0|_F)=0,
$$
which proves (1).


Finally, suppose that $p_g(\Gamma)>0$, we prove that Condition (2)  holds. Assume to the contrary that $W_0$ is not minimal, then by the Cone Theorem (see \cite[Theorem 3.7]{KM}), there exists a rational curve $\Gamma''$ on $W_0$ such that $(K_{W_0}\cdot \Gamma'')<0$. Since $W_0$ is relative minimal over $\Gamma$,  $\Gamma''$ is not contained in any fiber of $g$. Hence  $\Gamma''$ dominates $\Gamma$, but this is absurd. 
\end{proof}

Proposition \ref{pencil bound} implies the following corollary which plays a key role in  our main theorem.
\begin{cor}\label{base12} Let $X$ be a minimal projective $3$-fold of general type such that
$|K_X|$ is composed with a pencil of $(1,2)$-surfaces. Assume that one of the following holds:
\begin{enumerate}
\item $|K_X|$ is composed with an irrational pencil; or
\item $|K_X|$ is composed with a rational pencil and $p_g(X)\geq 21$; or
\item $X$ is Gorenstein, that is, $K_X$ is Cartier.
\end{enumerate}
Then there exists a minimal projective $3$-fold $Y$, being birational to $X$, such that $\Mov |K_{Y}|$ is base point free.
\end{cor}
\begin{proof}
Keep the notation in Subsection \ref{b setting}. Then there exists a resolution $\pi:W\to X$ such that $W$ admits a fibration structure $f: W\to \Gamma$ onto a smooth curve $\Gamma$ which is defined by $\Mov|\rounddown{\pi^*(K_X)}|=\Mov|K_W|$. By the construction, $\Mov|K_W|=|f^*H|$ for some base point free linear system $|H|$ on $\Gamma$. Furthermore we may assume that $f$ factors through its relative minimal model $g: W_0\to \Gamma$ and a birational morphism $\zeta:W\to W_0$.
Note that $\Mov |K_{W_0}|$ is also free since $\Mov |K_{W_0}|=|g^*H|$.
It suffices to show that $W_0$ is minimal and then one may simply take $Y=W_0$.

If $p_g(\Gamma)>0$, then $W_0$ is minimal by Lemma \ref{lem no E=minimal}.

If $\Gamma\cong \bP^1$ and the general fiber $F$ (or, equivalently, the minimal model $F_0$)  is a $(1,2)$-surface, then
$$
\pi^*(K_X)\sim_{\bQ} (p_g(X)-1)F+Z'
$$
for some effective $\bQ$-divisor $Z'$.
Since $p_g(X)\geq 21$, we have $p_g(X)-1\geq \frac{2}{\glct(F_0)}$ by Theorem \ref{12alpha new} and $p_g(X)-1>2K^2_{F_0}$. Hence by Proposition \ref{pencil bound}, there does not exist a  $\pi$-exceptional prime divisor $E_0$ on $W$ such that $(\pi^*(K_X)|_F\cdot E_0|_F)>0$, which means that $W_0$ is minimal by Lemma \ref{lem no E=minimal}. 

If $X$ is Gorenstein, then this is well-known (see \cite{CCZ06}) and we give a proof here. Consider  $\pi^*(K_X)|_F\leq K_W|_F=K_F$. Note that $\pi^*(K_X)|_F$ is a nef and big Cartier divisor and $\zeta_F^*(K_{F_0})$ is the positive part of the Zariski decomposition of $K_F$,
so $\pi^*(K_X)|_F\leq \zeta_F^*(K_{F_0})$. This implies that $$1\leq (\pi^*(K_X)|_F)^2\leq (\zeta_F^*(K_{F_0}))^2=1.$$
 Hence we have $(\pi^*(K_X)|_F)^2=(\zeta_F^*(K_{F_0}))^2=1$, and hence $\pi^*(K_X)|_F=\zeta_F^*(K_{F_0})$ by the uniqueness of the Zariski decomposition. Then $W_0$ is minimal by Lemma \ref{lem no E=minimal}.
\end{proof}

\section{\bf Proof of theorems}
Now we are prepared to prove the main results.

Let $X$ be a minimal projective $3$-fold of general type. We may always assume that $p_g(X)\geq 3$ since, otherwise, the Noether inequality automatically holds.
 We consider the canonical map $\varphi_1=\Phi_{|K_X|}$ which is non-trivial.  Set $d_X:=\dim \overline{\varphi_1(X)}$.
If $d_X=3$, then $K_X^3 \ge 2p_g(X)-6$ by \cite{Kob92} or \cite[Proposition 3.1]{MChen04JMSJ}.
We only need to consider the case that $d_X\leq 2$.

In fact,  we may assume that $p_g(X)\geq  5$, as $p_g(X)\leq 4$ was treated in \cite[Theorem 1.5]{MChen07}, but we will use weaker conditions as  $p_g(X)\geq  4$ or $p_g(X)\geq  3$ for the generality of the statements.

\subsection{Case $d_X=2$}\

In this subsection, we settle the case $d_X=2$. Firstly we consider the case when $|K_X|$ gives a fibration of curves of genus $>2$.
\begin{thm}\label{d2g3}
Let $X$ be a minimal projective $3$-fold of general type with $p_g(X)\geq 4$.
Suppose that $d_X=2$ and $|K_X|$ gives a fibration of curves of genus $>2$. Then
$$K_X^3\ge 2p_g(X)-4.$$
\end{thm}
\begin{proof}
Keep the notation in Subsection \ref{b setting} with $D=K_X$.
Let $S\in \Mov|\rounddown{\pi^*(K_X)}|$ be a general member. Note that $S$ is a nef divisor. By Lemma \ref{S^2=aC}, we have
$S|_{S}\equiv aC$ for an integer $a\ge p_g(X)-2\geq 2$ and $C$ is a general fiber of the restricted fibration
$f|_{S}: S\to f(S).$
Note that $C$ can be also viewed as a general fiber of $f$.
 One has
 \begin{align*} K_X^3 {}&= (\pi^*K_X)^3 \ge(\pi^*K_X \cdot S^2)= (\pi^*K_X|_S \cdot S|_S) \\{}&= a (\pi^*K_X|_S \cdot C) \ge (p_g(X)-2)( \pi^*K_X|_S \cdot C).
 \end{align*}
 It suffices to show that $(\pi^*K_X|_S \cdot C)\geq 2$.

 Denote by $S_0$ the minimal model of $S$ and $\sigma: S\to S_0$ the induced map.
Note that $K_S=(K_W+S)|_S\geq 2S|_S\equiv 2aC$, which implies that $K_S-2aC$ is pseudo-effective. Hence $K_{S_0}-2aC_0$ is also pseudo-effective, where $C_0=\sigma(C)$. Note that $C_0$ is a moving curve on $S_0$. If $C_0^2=0$, then $(K_{S_0}\cdot C_0)=2p_g(C_0)-2\geq 4.$ If $C_0^2\geq 1$, then $(K_{S_0}\cdot C_0)\geq 2aC_0^2\geq 2a\geq 4$. In any case,  we have $(\sigma^*K_{S_0}\cdot C)=(K_{S_0}\cdot C_0)\geq  4$.
On the other hand, by Corollary \ref{KXKS} and the fact that $\pi^*K_X\geq S$, $2\pi^*K_X|_S-\sigma^*K_{S_0}$ is $\bQ$-effective. Since $C$ is a moving curve,
$$(\pi^*K_X|_S \cdot C)\geq\frac{1}{2}(\sigma^*K_{S_0}\cdot C) \geq 2.
$$
The proof is completed.
\end{proof}

Then we consider the case when $|K_X|$ gives a fibration of curves of genus $2$.
\begin{thm}\label{d2g2}
Let $X$ be a minimal projective $3$-fold of general type with $p_g(X)\geq 3$.
Suppose that $d_X=2$ and that $|K_X|$ gives a fibration of curves of genus $2$. Then
$K_X^3\ge \frac{4}{3}p_g(X)-\frac{10}{3}.$
\end{thm}
\begin{proof}
Keep the notation in Subsection \ref{b setting} with $D=K_X$.
Let $S\in \Mov|\rounddown{\pi^*(K_X)}|$ be a general member. By Lemma \ref{S^2=aC}, we have
$S|_{S}\equiv aC$ for an integer $a\ge p_g(X)-2$ and $C$ is a general fiber of the restricted fibration
$f|_{S}: S\to f(S)$ with $p_g(C)=2$.
 Denote by $S_0$ the minimal model of $S$ and $\sigma: S\to S_0$ the induced map.
Note that $K_S=(K_W+S)|_S\geq 2S|_S\equiv 2aC$, which implies that $K_S-2aC\equiv G$ where $G=K_S-2S|_S$ is an effective integral divisor. Note that $p_g(S)\geq h^0(K_X|_S)+1\geq p_g(X)\geq 3$ by adjunction. Hence by Proposition \ref{vol S>8/3},
$$\vol(S)=K_{S_0}^2\geq \frac{8}{3}(2a-1). $$ On the other hand, by Corollary \ref{KXKS} and the fact that $\pi^*K_X\geq S$, $2\pi^*K_X|_S-\sigma^*K_{S_0}$ is $\bQ$-effective. Since both $\pi^*K_X|_S$ and $\sigma^*K_{S_0}$ are nef divisors,
$$K_X^3\geq (\pi^*K_X|_S)^2\geq\frac{1}{4}(\sigma^*K_{S_0})^2 \geq \frac{2}{3}(2a-1)\ge \frac{4}{3}p_g(X)-\frac{10}{3}.
$$
The proof is completed.
\end{proof}

\begin{remark} In the first version of this paper, we used the so-called ``feasible resolution'' to prove Theorem \ref{d2g2}.  A keen observation due to Yong Hu, who suggests to make use of Theorem
\ref{kaw ext}, greatly helps us in forming the present proof of Theorem \ref{d2g2}.
\end{remark}

\subsection{Case $d_X=1$}\

In this subsection, we settle the case $d_X=1$. Firstly, we consider the case when $|K_X|$ is neither composed with a pencil of $(1,1)$-surfaces nor $(1,2)$-surfaces.

\begin{thm}\label{not 11 not 12} Let $X$ be a minimal projective $3$-fold of general type with $p_g(X)\geq 3$. Suppose that $d_X=1$ and $|K_X|$ is neither composed with a pencil of $(1,1)$-surfaces nor $(1,2)$-surfaces. Then $$K_X^3> 2p_g(X)-6.$$
\end{thm}

\begin{proof}
Keep the notation in Subsection \ref{b setting}. Take a resolution $\pi: W\to X$, we have a morphism $f: W\to \Gamma$ defined by $M=\Mov|\rounddown{\pi^*(K_X)}|$. Denote by $F$ the general fiber of $f$ which is a smooth surface of general type and $F_0$ the minimal model of $F$ with the induced map $\sigma:F\to F_0$.
Since $p_g(W)=p_g(X)>0$, $p_g(F)>0$ by adjunction. In this case, $K^2_{F_0}\geq 2$ by the Noether inequality since $F$ is neither a $(1,1)$-surface nor a $(1,2)$-surface.
Note that $\pi^*(K_X)\geq M\equiv aF$ for some integer $a\geq p_g(X)-1$.

If $p_g(\Gamma)=0$, then by Corollary \ref{KXKS} and the fact that $\pi^*(K_X)-aF$ is $\bQ$-effective, $$\Big(1+\frac{1}{a}\Big)\pi^*(K_X)|_F-\sigma^*(K_{F_0})$$ is $\bQ$-effective. If $p_g(\Gamma)>0$, then by Lemma \ref{lem no E=minimal}(3), $\pi^*(K_X)|_F= \sigma^*(K_{F_0})$.
Hence, in both situations,
we have
\begin{align*}
K_X^3{}&\geq a(\pi^*(K_X)|_F)^2 \geq \frac{a^3}{(a+1)^2}K_{F_0}^2\\
{}&\geq \frac{2a^3}{(a+1)^2}>2a-4\geq 2p_g(X)-6.
\end{align*}
The proof is completed.
\end{proof}

Then we consider the case when $|K_X|$ is composed with a pencil of $(1,1)$-surfaces.
\begin{thm}\label{1,1} Let $X$ be a minimal projective $3$-fold of general type with $p_g(X)\geq 3$. Suppose that $d_X=1$ and $|K_X|$ is composed with a pencil of $(1,1)$-surfaces. Then $$K_X^3> 2p_g(X)-6.$$
\end{thm}
\begin{proof}
Keep the notation in Subsection \ref{b setting}. Take a resolution $\pi: W\to X$, we have a morphism $f: W\to \Gamma$ defined by $M=\Mov|\rounddown{\pi^*(K_X)}|$. Denote by $F$ the general fiber of $f$ which is a $(1,1)$-surface of general type and $F_0$ the minimal model of $F$ with the induced map $\sigma:F\to F_0$. Note that $M\equiv aF$ for some integer $a\geq p_g(X)-1$.

Since $p_g(F)=1$, the sheaf $f_*\omega_{W}$ is a line bundle on $\Gamma$ of degree $a$, as $f_*\omega_{W}$ is torsion-free by \cite[Theorem 2.1]{Kol} and its rank is $p_g(F)$. Since $q(F)=0$ (see \cite[Theorem 11]{Bom73}), we get
\begin{align*}
{}&q(X)=h^2(W, \omega_{W})=p_g(\Gamma),\\
{}&h^2(\OO_{X})=h^1(W, \omega_{W})=h^1(\Gamma, f_*\omega_{W}).
\end{align*}

{}We claim that $\chi(\omega_{X})\geq p_g(X)-1$. 

First we consider the case $p_g(\Gamma)>1$. If $h^1(\Gamma, f_*\omega_{W})=0$, then $\chi(\omega_{X})=p_g(X)+p_g(\Gamma)-1$. If $h^1(\Gamma, f_*\omega_{W})>0$, then
 $$\chi(\omega_X)=\chi(f_*\omega_{W})+p_g(\Gamma)-1=\deg(f_*\omega_{W})\geq 2p_g(X)-2$$
by Clifford's theorem.

Next we consider the case $p_g(\Gamma)\leq 1$. Since $f_*\omega_{W}$ is a line bundle of degree $a>0$, we see $h^2(\OO_X)=0$. Thus we get
\begin{equation*} \chi(\omega_{X})=p_g(X)+q(X)-h^2(\OO_{X})-1\geq p_g(X)-1.\end{equation*}
In summary, $\chi(\omega_{X})\geq p_g(X)-1$ holds for all cases.

We claim that $|2K_W|$ distinguishes two general fibers $F_1$ and $F_2$ of $f$, that is, the images of $F_1$ and $F_2$ under the rational map defined by $|2K_W|$ are different. When $p_g(\Gamma)=0$, this is clear. Assume that $p_g(\Gamma)>0$,
recall that we can write $\pi^*(K_X)\equiv aF+E$ for an integer $a\geq p_g(X)\geq 3$ and an effective $\bQ$-divisor $E$. After replacing $W$ with a higher model, we may assume that $E$ has simple normal crossing support. Then by the Kawamata--Viehweg vanishing theorem,
$$
H^1(K_W+\roundup{{\pi^*(K_X)-tE-F_1-F_2}})=0
$$
where $t=\frac{2}{a}<1$ and $\pi^*(K_X)-tE-F_1-F_2\equiv (1-t)\pi^*(K_X)$ is nef and big.
Hence the natural map
$$
H^0(K_W+\roundup{{\pi^*(K_X)-tE}})\to H^0(F_1, D_1)\oplus H^0(F_2, D_2).
$$
is surjective, where $D_i=(K_W+\roundup{{\pi^*(K_X)-tE}})|_{F_i}$ is effective for $i=1,2$. This implies that $|K_W+\roundup{{\pi^*(K_X)-tE}}|$ distinguishes $F_1$ and $F_2$. By adding an effective divisor not containing $F_1$ and $F_2$, we know that $|2K_W|$ distinguishes $F_1$ and $F_2$.

By the plurigenus formula of Reid (\cite{YPG}), we have
$$P_2(X):=h^0(X, 2K_X)\ge \frac{1}{2}K_X^3-3\chi(\OO_{X})\geq
\frac{1}{2}K_X^3+3p_g(X)-3. $$
Set $|M_2|:=\Mov |2K_{W}|$. After replacing $\pi$ with a further birational modification, we may assume that $|M_2|$ is also base point free and defines a morphism $\varphi_2$ on $W$. Pick a general member $S_2\in |M_2|$. We consider the natural restriction map $\nu_2$:
$$H^0(W, S_2)\overset{\nu_2}\lra V_2\subseteq H^0(F, S_2|_F)\subseteq
H^0(F, 2K_F),$$
where $V_2$ denotes the image of $\nu_2$ as a ${\Bbb C}$-subspace of $H^0(F, S_2|_F)$. Since $h^0(2K_F)=3$, we see that
$1\le\dim_{\Bbb C}V_2\le 3$. Denote by $\Lambda_2$ the linear system on $F$ corresponding to $V_2$. We have $\dim\Lambda_2=\dim_{\Bbb C}(V_2)-1$.
\medskip

\noindent {\bf Case 1}. $\dim \Lambda_2=2$.

In this case, since $\Lambda_2$ is a sub-linear system of $|2K_F|$ of maximal dimension, we see that $\Lambda_2\supset \Mov |2K_F|$ and thus $\varphi_{2}|_F$ coincides with $\varphi_{2, F}$, the morphism defined by $\Mov |2K_F|$ on $F$. It is well-known that $\varphi_{2, F}$ is generically finite of degree $4$ since $F$ is a $(1,1)$-surface (see, for example, \cite{Xiao90}). Since $\varphi_2$ distinguishes general fibers of $f$, $\varphi_{2}$ is generically finite of degree $4$. 

Set $L_2:=S_2|_{S_2}$. We consider the natural map
$$H^0(W, S_2)\overset{\nu_2'}\lra {V'_2}\subseteq H^0(S_2, L_2),$$
where ${V'_2}$ is the image of $\nu_2'$ with $\dim_{\Bbb C}V'_2=h^0(W, S_2)-1=P_2(X)-1$. Denote by $\Lambda'_2$ the sub-linear system of $|L_2|$ corresponding to ${V'_2}$.
Then $\Lambda'_2$ defines a generically finite map of degree $4$ on $S_2$.
By \cite[Lemma 2.2(ii)]{MChen04JMSJ},
$$L_2^2\ge 4(\dim {\Lambda}'_2-1)\ge 4(P_2(X)-3).$$
Therefore we have
$$8K_X^3\ge S_2^3=L_2^2\ge4(P_2(X)-3)
\ge 4 \Big(\frac{1}{2}K_X^3+3p_g(X)-6\Big),$$
which implies that
$K_X^3\ge 2p_g(X)-4. $
\medskip

\noindent{\bf Case 2}. $\dim\Lambda_2=1$.

In this case, $\dim\varphi_{2}(F)=1$ and $\dim\varphi_{2}(W)=2$. Taking the Stein factorization of $\varphi_{2}$, we get an induced fibration $f_2:W\to \Sigma$ where $\Sigma$ is a normal projective surface. Let $C'$ be a general fiber of $f_2$. We can also view $C'$ as a general fiber of $f_2|_F: F\to f_2(F)$. 
Since $F$ is a $(1,1)$-surface and $C'$ is a moving curve which comes from a fibration of $F$,
by \cite[Lemma 2.4]{exp3}, we have $(\sigma^*(K_{F_0})\cdot C')\ge 2$.

If $p_g(\Gamma)=0$, then by Corollary \ref{KXKS} and the fact that $\pi^*(K_X)-aF$ is $\bQ$-effective, $$\Big(1+\frac{1}{a}\Big)\pi^*(K_X)|_F-\sigma^*(K_{F_0})$$ is $\bQ$-effective. If $p_g(\Gamma)>0$, then by Lemma \ref{lem no E=minimal}(3), $\pi^*(K_X)|_F= \sigma^*(K_{F_0})$.

Hence we always have
$$(\pi^*(K_X)\cdot C')=(\pi^*(K_X)|_F\cdot C')\ge\frac{a}{a+1}(\sigma^*(K_{F_0})\cdot C')\ge \frac{2a}{a+1},$$
where $a\geq p_g(X)-1$.

On the general surface $S_2$, by Lemma \ref{S^2=aC}, we may write
$S_2|_{S_2}\equiv a_2C'$
for an integer $a_2\ge P_2(X)-2$. Noting that
$$(\pi^*(K_X)|_{S_2}\cdot C')=(\pi^*(K_X)\cdot C')\ge\frac{2a}{a+1}$$
and $2\pi^*(K_X)\ge S_2$, we have
\begin{align*}
4K_X^3{}&\ge (\pi^*(K_X)\cdot S_2\cdot S_2)
 = a_2(\pi^*(K_X)|_{S_2}\cdot C')\\
{}&\ge\frac{2a}{a+1}(P_2(X)-2)
\ge \frac{2a}{a+1}\Big(\frac{1}{2}K_X^3+3p_g(X)-5\Big).
\end{align*}
Thus it follows that
\begin{align*}
K_X^3&{}\ge\frac{6a}{3a+4}p_g(X)-\frac{10a}{3a+4}\\
{}&=2p_g(X)-6+\frac{8a-8p_g(X)+24}{3a+4}\\
{}&> 2p_g(X)-6. 
\end{align*}
\medskip

\noindent{\bf Case 3}. $\dim \Lambda_2=0$.

In this case, $\varphi_{2}$ is trivial on $F$, which means that $\varphi_{2}$ and $\varphi_{1}$ induce the same fibration $f:W\to \Gamma$ after taking the Stein factorizations. So we may write
$$
2\pi^*(K_X)\sim \sum_{i=1}^{a_2}F_i+E_2'
\equiv a_2F+E_2',$$
where the surfaces $F_i's$ are smooth fibers of $f$, $E_2'$ is an effective ${\Bbb Q}$-divisor, and $a_2\ge P_2(X)-1$. 

If $p_g(\Gamma)=0$, then by Corollary \ref{KXKS} and the fact that $2\pi^*(K_X)-a_2F$ is $\bQ$-effective, $$\Big(1+\frac{2}{a_2}\Big)\pi^*(K_X)|_F-\sigma^*(K_{F_0})$$ is $\bQ$-effective. If $p_g(\Gamma)>0$, then by Lemma \ref{lem no E=minimal}(3), $\pi^*(K_X)|_F= \sigma^*(K_{F_0})$.

Hence
 we have
\begin{align*}
2K_X^3{}&=(\pi^*(K_X)^2\cdot 2\pi^*(K_X))\geq (\pi^*(K_X)^2\cdot a_2F)\\
{}&= a_2(\pi^*(K_X)|_F)^2\ge\frac{a_2^3}{(a_2+2)^2}K_{F_0}^2\\
{}&= a_2-4+\frac{12a_2+16}{(a_2+2)^2}\\
{}&>P_2(X)-5\\
{}&\geq \frac{1}{2}K_X^3+3p_g(X)-8,
\end{align*}
which gives
$$K_X^3>2p_g(X)-\frac{16}{3}. $$

Combining all above cases, the inequality is proved.
\end{proof}

Here we remark that in the above proof, we further consider the map $|2K_X|$. The main reason this method works for this case is that, after restricting to the fiber $F$, $|2K_F|$ gives a generically finite map of degree $4$ (which is the key point in solving Case (1)). However, this method fails for pencils of $(1,2)$-surfaces, because in this case $|2K_F|$ gives a generically finite map of degree $2$, which is too small for our desired inequality.





Finally we consider the case when $|K_X|$ is composed with a pencil of $(1,2)$-surfaces and $\Mov|K_X|$ is free. Recall the following lemma, which is a generalization of \cite[Lemma 4.6]{MChen04JMSJ} with a simplified proof.

\begin{lem}[{cf. \cite[Lemma 4.6]{MChen04JMSJ}}]\label{2d}Let $V$ be a projective $3$-fold with at worst terminal singularities. Suppose that $p_g(V)\geq 2$, and there is a fibration $\phi: V\to T$ to a smooth curve $T$ whose general fiber $F$ is a smooth surface with $q(F)=0$ (e.g., $F$ is a $(1,2)$-surface). Fix a general fiber $F^{(0)}$, then for a general fiber $F$, the natural restriction map $
H^0(V, K_V+2F^{(0)})\to H^0(F, K_{F})$
is surjective and therefore $\Phi_{|K_V+2F^{(0)}|}|_{F}=\Phi_{|K_{F}|}$.
\end{lem}
\begin{proof}
Note that, since the nature of the statement is invariant under birational equivalence,
we may assume $V$ to be smooth. Since $q(F)=0$, we have $R^1\phi_*\omega_V=0$. Then, by Koll\'ar's vanishing \cite[Theorem 2.1]{Kol},
$$
h^1(V, K_V+2F^{(0)}-F )=h^1(T, \phi_*(\omega_V)\otimes \OO_T(2t_0-t ))=0,
$$
where $t_0=\phi(F^{(0)})\in T$ and $t =\phi(F )\in T$. Hence the natural restriction map
$$
H^0(V, K_V+2F^{(0)})\to H^0(F, K_{F })
$$
is surjective.
\end{proof}

\begin{thm}\label{last} Let $X$ be a minimal projective $3$-fold of general type with $p_g(X)\geq 4$. Assume that $d_X=1$ and $|K_X|$ is composed with a pencil of $(1,2)$-surfaces. Moreover, assume that $\Mov|K_X|$ is base point free. Then $$K_X^3\geq \frac{4}{3}p_g(X)-{\frac{10}{3}}.$$
\end{thm}

\begin{proof}
 {\bf Step 0. }Overall settings.

By the assumption, $\Mov|K_X|$ induces a fibration $f:X\to \Gamma$. Denote by $F_X$ the general fiber of $f$, which is a minimal $(1,2)$-surface since $X$ is minimal.
Fix a general fiber $F_{X, 0}$ of $f$ and consider the linear system $|D|=|K_X+2F_{X,0}|$. By Lemma \ref{2d}, $\Phi_{|D|}|_{F_X}=\Phi_{|K_{F_X}|}$ which defines a genus $2$ fibration $\tilde{F}_X\to \bP^1$, where $\tilde{F}_X$ is any higher model of $F_X$ resolving the base point of $|K_{F_X}|$. This means that $\dim\Phi_{|D|}(X)=2$, and $|D|$ gives a fibration of curves of genus $2$.

Take a resolution $\pi: W\rightarrow X$ such
that $|M|=\Mov|\rounddown{\pi^*D}|$ is base point free.
Let $W\overset{\beta}\longrightarrow \Sigma\overset{s}\longrightarrow Z$
be the Stein factorization of $\gamma=\pi\circ \Phi_{|D|}$ with $Z=\gamma(W)\subseteq
\bP^{h^0(D)-1}$. We have the following commutative
diagram:
$$\xymatrix@=4.5em{
&W\ar[ld]_{f'}\ar[d]_\pi \ar[dr]^{\gamma} \ar[r]^\beta& \Sigma\ar[d]^s\\
\Gamma &X \ar[l]_f \ar@{-->}[r]^{\Phi_{|D|}} & Z}
$$

Let $S\in \Mov|\rounddown{\pi^*(D)}|$ be a general member. By Lemma \ref{S^2=aC}, we have
$S|_{S}\equiv aC$ for an integer $a\ge h^0(D)-2$ and $C$ is a general fiber of the restricted fibration
$\beta|_{S}: S\to \beta(S)$ with $p_g(C)=2$.
Note that $C$ can be viewed as a general fiber of $\beta$.
 Denote by $S_0$ the minimal model of $S$ and $\sigma: S\to S_0$ the induced map. Denote by $F$ the general fiber of $\pi\circ f$ and $F^{(0)}$ the fiber corresponding to $F_{X, 0}$.
 Note that by the construction, $|S|$ gives the canonical pencil $F\to \bP^1$ on $F$, hence $(S|_F)^2=0$. This means that $(F|_S\cdot S|_S)=0$. Hence $F|_S=C$ is a general fiber of $\beta|_S$.
\medskip

\noindent {\bf Step 1.} A general inequality.

Since $K_W+2F^{(0)}\geq S$, we have $K_S=(K_W+S)|_S\geq (2S-2F^{(0)})|_S\equiv (2a-2)C$, which implies that $K_S-(2a-2)C\equiv G$ where $G$ is an effective integral divisor. Note that
$$p_g(S)> h^0(S, K_W|_S)\geq h^0(X, K_X)-h^0(X, K_X-S)= p_g(X)\geq 4.$$ Hence by Proposition \ref{vol S>8/3},
$$\vol(S)=K_{S_0}^2\geq \frac{8}{3}(2a-3). $$ On the other hand, by Corollary \ref{KXKS} and the fact that $\pi^*(D)\geq S$, $\pi^*(K_X+D)|_S-\sigma^*K_{S_0}$ is $\bQ$-effective. Note that both $\pi^*(K_X+D)|_S$ and $\sigma^*K_{S_0}$ are nef divisors. Hence
\begin{align*}
4K_X^3+16{}&= (\pi^*(K_X+D)^2\cdot \pi^*(D))\geq (\pi^*(K_X+D)^2\cdot S)\\
{}&\geq (\sigma^*K_{S_0})^2\geq \frac{8}{3}(2a-3).
\end{align*}
This means that
\begin{align}
K_X^3\geq \frac{4}{3}a-6.\label{last ineq}
\end{align}

Now we estimate the value of $a$.
By \cite[Lemma 4.5]{MChen04JMSJ}, $0\leq p_g(\Gamma)\leq 1$ and $h^2(\OO_W)=h^2(\OO_X)\leq 1-p_g(\Gamma)$.

If $p_g(\Gamma)=1$, then $h^1(K_X)=h^2(\OO_X)=0$. Also by the proof of Lemma \ref{2d}, $h^1(K_X+F_{X,0})=h^1(K_W+F^{(0)})=0$. Arguing by exact sequences, it is easy to see that
$$
h^0(D)=h^0(K_X+2F_{X,0})=h^0(K_X)+2h^0(K_{F_{X,0}})=p_g(X)+4.
$$
Here the first and the third equalities follow by definition, and the second follows from the exact sequences
 $$
 0\to H^0(X, K_X+F_{X,0})\to H^0(X, K_X+2F_{X,0})\to H^0(F_{X,0}, K_{F_{X,0}})\to 0
 $$
 and
 $$
 0\to H^0(X, K_X)\to H^0(X, K_X+F_{X,0})\to H^0(F_{X,0}, K_{F_{X,0}})\to 0.
 $$

If $p_g(\Gamma)=0$, then $|K_X|$ is composed with a rational pencil and $K_X\geq (p_g(X)-1)F_{X,0}$. On the other hand, $|D|=|K_X+2F_{X,0}|$ is not composed with a pencil by the construction, hence
$$
h^0(D)=h^0(K_X+2F_{X,0})\geq h^0((p_g(X)+1)F_{X,0})+1=p_g(X)+3.
$$

Hence, in both situations, we have
$$
a\ge h^0(D)-2\geq p_g(X)+1.
$$
Note that $a$ is an integer by the construction. If $a\geq p_g(X)+2$, we can get the desired inequality by Inequality \eqref{last ineq}.
\medskip

\noindent {\bf Step 2.} The case with $a=p_g(X)+1$.

{}From now on we assume that $a=p_g(X)+1$. In this case, by Inequality \eqref{last ineq} we can only get a weaker inequality
\begin{align*}
K_X^3\geq \frac{4}{3}p_g(X)-\frac{14}{3}.
\end{align*}
So we have to study the structure of $X$ in more details to get a better inequality. Since the equality
$$
a= h^0(D)-2= p_g(X)+1
$$
holds, we see  $p_g(\Gamma)=0$. The first equality implies that $\Sigma\simeq Z\subseteq \bP^{h^0(D)-1}$ is a surface of minimal degree by Lemma \ref{S^2=aC}. Note that $h^0(D)\geq p_g(X)+3\geq 7$. By \cite[Exercise IV.18.4]{Beauville}, $Z$ is either $\bP(1,1,a)$ (i.e., the cone over a rational curve of degree $a$), or the $r$-th Hirzebruch surface $\bF_r$ for some $r\geq 0$.
In particular, $\Sigma\simeq Z$ is normal.
\smallskip

\noindent {\bf Sub-step 2.1.} We claim that  $Z=\bP(1,1,a)$.

Assume, to the contrary, that $Z=\bF_r$ for some $r$, then by \cite[Exercise IV.18.4]{Beauville}, the embedding $Z\subseteq \bP^{h^0(D)-1}$ is given by $|\sigma_0+(r+k)\ell|$, where $r+2k=h^0(D)-2$, $k\geq 1$, $\sigma_0$ is the section with self-intersection $-r$, and $\ell$ is a ruling. Hence $S\sim\gamma^*(\sigma_0+(r+k)\ell)$. Note that $F$ and $\gamma^*(\ell)$ both give rational pencils on $W$. If they give different pencils, then $\gamma^*(\ell)|_F$ is a moving curve on $F$. But then
\begin{eqnarray}
1&=&(\pi^*(K_X)\cdot \pi^*(K_X+2F_X)\cdot F)\notag \\
&\geq& (\pi^*(K_X)|_F\cdot S|_F)
\geq (\pi^*(K_X)|_F\cdot (k+r)\gamma^*(\ell)|_F)\notag\\
&=&(\pi|_F^*(K_{F_X})\cdot (k+r)\gamma^*(\ell)|_F)\geq k+r.\label{l=F}
\end{eqnarray}
This is absurd since $k+r\geq \frac{1}{2}(h^0(D)-2)>1$ by the assumption. Hence $F$ and $\gamma^*(\ell)$ give the same rational pencil and $F\sim\gamma^*(\ell)$ since they are irreducible. But by the construction, $S\geq (p_g(X)+1)F$, which implies that $|\sigma_0+(r+k-p_g(X)-1)\ell|\neq \emptyset.$ This is also absurd since
$r+k<r+2k=h^0(D)-2=p_g(X)-1.$
\smallskip

\noindent {\bf Sub-step 2.2.} Construct a special model $W$.

To proceed further discussion, we describe an explicit way to construct a special resolution $W$ resolving the base locus of $\Mov|{D}|$.

Firstly we fix some notation. Since $\Gamma=\bP^1$, we do not need to fix $F_{X,0}$ and we will just write $F_X$ instead of $F_{X,0}$. Consider a general fiber $F_X$ of $f$, which is a minimal $(1,2)$-surface. It is known that $|K_{F_X}|$ has a unique base point $P$ (see \cite[Section 2]{Hori76}). Denote by $\eta: F'_X\to F_X$ the blow-up at $P$ of $|K_{F_X}|$, and $\mathfrak{e}_0$ the exceptional divisor, then $\Mov|K_{F'_X}|=|\eta^*(K_{F_X})-\mathfrak{e}_0|$ is base point free (see \cite[Section 2, p. 129]{Hori76}).

In the following, we will  construct a special resolution $W$ resolving the base locus of $\Mov|{D}|$ (as in the setting of Step 0), with a nice property that  a general fiber of $f':W\to \bP^1$ is isomorphic to $F'_X$. This is possible since, roughly speaking, after blowing up once along the section of $X\to \bP^1$ coming from the unique base point of $|K_{F_X}|$ for general fiber $F_X$, the base locus of the strict transform of $\Mov|{D}|$ does not dominant $\bP^1$, and hence we can further resolve it avoid changing the general fiber. We explain this construction in details as the following.

Note that by Lemma \ref{2d}, $P=\Bs|K_{F_X}|= (\Bs(\Mov|K_X+2F_X|))|_{F_X}$. Therefore, there is a curve $\mathfrak{B}\subseteq \Bs( \Mov|K_X+2F_X|)$ such that $\mathfrak{B}|_{F_X}=P$ for a general fiber $F_X$.
In the following, for $i=1,2$, $F_i$ denotes the general fiber of the composition $X_i\to X\to \bP^1$.

Take a resolution $X_1\to X$ which resolves (isolated) singularities of $X$ and singularities of $\mathfrak{B}$. Note that this resolution does not affect the general fiber, i.e., $F_1\cong F_X$. Denote by $\mathfrak{B}_1$ the strict transform of $\mathfrak{B}$ on $X_1$, which is a smooth curve. Then take $X_2\to X_1$ to be the blow-up along $\mathfrak{B}_1$. Then $F_2\cong F'_X$ and, by Lemma \ref{2d},
$\Mov|K_{X_2}+2F_2||_{F_2}= \Mov|K_{F_2}|$ which is base point free and defines a rational pencil.
Hence the base locus of $\Mov|K_{X_2}+2F_2|$ does not intersect with a general fiber $F_2$. Take a resolution $W\to X_2$ to resolve the base locus of $\Mov|K_{X_2}+2F_2|$. Then
$$|S|=\Mov|\pi^*(K_X+2F_X)|=\Mov|K_{W}+2F|$$
 is base point free and $F\cong F'_X$, where $\pi:W\to X$ is the induced map and $F$ is the general fiber of $f'=f\circ \pi:W\to \bP^1$.
Hence $|S|$ defines a morphism $\gamma:W \to Z=\bP(1,1,a)$.  The advantage of such an  explicit construction of $W$ is that the general fiber of $f'$ is isomorphic to $F'_X$, which is easy to describe and has a clear structure.
\smallskip

\noindent {\bf Sub-step 2.3.} We claim that $\gamma:W\to Z$ factors through $\bF_a$.

 In fact, take a resolution $\alpha: Y\to W$ such that $Y\to Z$ factors through $\bF_a$ and we have the following commutative
diagram:

$$\xymatrix@=4em{
Y\ar[d]_\alpha \ar[r]^{h}& \bF_a\ar[d]_\psi\\
 W\ar[r]^{\gamma} & \bP(1,1,a)}
$$
By abuse of notation, without any confusion, we still denote by $\sigma_0$ the negative section of $\bF_a$ and $\ell$ the ruling (which were previously used for  $\bF_r$). Denote by $\bar{\ell}$ the ruling of $\bP(1,1,a)$ and we have $S\sim \gamma^*(a \bar{\ell})$. Note that $\alpha^*F$ and $h^*\ell$ both define rational pencils on $Y$, arguing as Inequality \eqref{l=F}, it is easy to see that
$\alpha^*F\sim h^*\ell$. Note that $h$ can be defined by the linear system $|h^*(\sigma_0+(a+1)\ell)|$ as $\sigma_0+(a+1)\ell$ is very ample on $\bF_a$. On the other hand,
$$h^*(\sigma_0+(a+1)\ell)\sim h^*\psi^*(a \bar{\ell})+h^*\ell\sim \alpha^*(S+F).$$
This means that $h=\Phi_{|h^*(\sigma_0+(a+1)\ell)|}$ factors through $\Phi_{|S+F|}$. Since $h(Y)\cong \bF_a$ and $|S+F|$ is base point free on $W$,
$\Phi_{|S+F|}$ defines a morphism $g:W\to \bF_a$ onto $\bF_a$, which proves the claim.

Also we know that $F\sim g^*(\ell)$. In particular, we have the following commutative
diagram:
$$\xymatrix@=4em{
W\ar[d]_\pi\ar[dr]^{f'} \ar[r]^g& \bF_a\ar[d]\\
X \ar[r]^f & \bP^1}
$$

Denote $g^*(\sigma_0)=B$ which is an effective Cartier divisor on $W$.
We can write
 $$\pi^*K_X+2F\sim S+E''\sim aF+B+E''$$
 for an effective $\bQ$-divisor $E''$ and
 $$K_{W}=\pi^*K_X+E_\pi$$
 for an effective $\pi$-exceptional $\bQ$-divisor $E_\pi$.
 \smallskip

\noindent {\bf Sub-step 2.4.} Two distinguished components in $B$ and $E''$.
 \begin{claim}\label{E0D0} The following statements hold:
 \begin{enumerate}
 \item there exists a unique  $\pi$-exceptional prime divisor $E_0$ on $W$ such that $E_0$ dominates $\bF_a$. Moreover, 
 $(E_0\cdot C)=1$ and $\coeff_{E_0}E''=\coeff_{E_0}E_\pi=1$, where $C$ is a general fiber of $g$;
\item there exists a unique prime divisor $D_0$ which is a component of $B$ such that $(D_0\cdot E_0\cdot F)=1$, $\coeff_{D_0}B=1$, and $(\pi^*(K_X)\cdot (B-D_0)\cdot F)=0$. Here note that $D_0\neq E_0$.
 \end{enumerate}
 \end{claim}
\begin{proof}
(1) Consider a general fiber $F$ of $f'$. By the construction of $W$, $\pi_F=\pi|_F: F\to F_X$ is the blow-up at the unique base point of $|K_{F_X}|$. Under the circumstance of no confusion, denote by $\mathfrak{e}_0$ the $\pi_F$-exceptional divisor on $F$.
Then $\mathfrak{e}_0$ is contained in the exceptional locus of $\pi$. As the general fiber $F$ moves, $\mathfrak{e}_0$ forms a  $\pi$-exceptional prime divisor $E_0$ such that $E_0|_F=\mathfrak{e}_0$ for a general $F$. 
In other words, $E_0$ is just the strict transform of the exceptional divisor of $X_2\to X_1$ on $W$. This implies that 
 $\coeff_{E_0}E''\geq 1$ and $\coeff_{E_0}E_\pi \geq 1$.
Consider a general fiber $C$ of $g$ with $C\subseteq F$, by the construction, $C$ can be also viewed as a general fiber of $\Phi_{|K_F|}:F\to \bP^1$.
It is clear that, in this case, $C\sim \pi_F^*K_{F_X}-\mathfrak{e}_0$.
Hence $(E_0\cdot C)=(E_0|_F\cdot C)=(\mathfrak{e}_0\cdot (\pi_F^*K_{F_X}-\mathfrak{e}_0))=1$.
In particular, $E_0$ dominates $\bF_a$ birationally as $C$ is a general fiber of $g$.
On the other hand, note that $C$ is of genus $2$ and $(S\cdot C)=(F\cdot C)=0$, we have
\begin{align*}
(E''\cdot C){}&=(\pi^*K_X\cdot C)=(\pi_F^*(K_{F_X})\cdot C)=1\\
(E_\pi\cdot C){}&=(K_W\cdot C)-(\pi^*K_X\cdot C)=1.
\end{align*}
This shows that $\coeff_{E_0}E''=\coeff_{E_0}E_\pi= 1$ and, since $\Supp(E_\pi)$ contains all $\pi$-exceptional divisors, there is no any other $\pi$-exceptional divisor dominating $\bF_a$.

(2) 
Note that $(B\cdot C)=((S-aF)\cdot C)=0$, which implies that $E_0\not \subseteq \Supp(B).$ So $D_0\neq E_0$ if exists. Also, for any component $D_1$ of $B$, we have $(D_1\cdot E_0\cdot F)\geq 0$. On the other hand, $(B\cdot E_0\cdot F)=(S\cdot E_0\cdot F)=(C\cdot E_0|_F)=1$. As $B$ is Cartier, there exists a unique component $D_0$ in $B$ such that $(D_0\cdot E_0\cdot F)>0$, moreover, $(D_0\cdot E_0\cdot F)=1$ and $\coeff_{D_0}B=1$. For the last statement, note that $B|_F\sim S|_F$ where the latter one gives a pencil of genus $2$ curves on $F$. Hence $B|_F$ is a special fiber of this pencil and $D_0$ is the unique component such that $D_0|_F$ intersects $\mathfrak{e}_0$. 
Hence ${\pi_F}_*(D_0|_F)$ passes through the unique base point $P$ of $|K_{F_X}|$ which implies
that $(K_{F_X}\cdot {\pi_F}_*(D_0|_F))\geq 1$. On the other hand, since
$B|_F\sim S|_F \sim C\sim \pi_F^*K_{F_X}-\mathfrak{e}_0$,
$$(K_{F_X}\cdot {\pi_F}_*(B|_F))= K_{F_X}^2=1.$$
This implies that
$(K_{F_X}\cdot {\pi_F}_*((B-D_0)|_F))=0$.
Hence
\begin{align*}
(\pi^*(K_X)\cdot (B-D_0)\cdot F){}&=(\pi^*(K_X)|_F\cdot (B-D_0)|_F)\\{}&=(\pi_F^*(K_{F_X})\cdot (B-D_0)|_F)=0.
\end{align*}
The claim is proved.
\end{proof}

\noindent {\bf Sub-step 2.5.}  ``Weak pseudo-effectivity'' of $3\pi^*K_X-(a-6)F$.

By \cite[Lemma 7.3]{Vie83}, there is a smooth birational model $\psi':\Sigma'\to \bF_a$ and a resolution $W'$ of $W\times_{\bF_a}\Sigma'$ giving a commutative diagram
$$\xymatrix@=4em{
W'\ar[d]_{\pi'} \ar[r]^{g'}& \Sigma'\ar[d]^{\psi'}\\
W \ar[r]^g & \bF_a}
$$
such that every $g'$-exceptional divisor is $\pi'$-exceptional. We may assume that $E'_0$ is smooth by taking further modification, where $E'_0$ is the strict transform of $E_0$ on $W'$. We may find an ample Cartier divisor $A_{{\Sigma'}}$ on ${\Sigma'}$ with the form $A_{{\Sigma'}}=\psi'^*A-E_{{\Sigma'}}$, where  $A$ is an ample Cartier divisor on $\bF_a$ and $E_{{\Sigma'}}$ is an effective $\psi'$-exceptional divisor on ${\Sigma'}$. We may write $A=t_1\ell+t_2\sigma_0$ for some positive integers $t_1>at_2$.

Firstly we construct some divisor whose support does not contain $E_0$.
\begin{claim}\label{claim D-E}
For any integer $m>0$, there exists an integer $c>0$ and an effective divisor
$$H_m\sim cmK_{W/\bF_a}+cmE_0+cg^*A$$
such that $E_0\not\subseteq \Supp(H_m)$.
\end{claim}

\begin{proof}
By Theorem \ref{thm wp} and the fact that $E'_0$ is smooth, for any integer $m>0$,
$$
\FF_m=g'_*\OO_{W'}(mK_{W'/\Sigma'}+mE'_0 )
$$
is weakly positive. In particular, there is a positive integer $c$ (taking $\alpha=1$ in the definition of weak positivity) such that
$$
\GG_m=({S}^c\FF_m)^{**}\otimes \OO_{\Sigma'}(cA_{\Sigma'})
$$
is generically globally generated on $\Sigma'$. It is clear that
$\GG_m\otimes k(y)$ corresponds to a base point free linear system on $C_y={g'}^{-1}(y)$ for a general point $y\in \Sigma'$. Hence, by the generic global generation, we can find a global section of $\GG_m$ not vanishing on the point $E'_0\cap C_y$ for a general $y$. Note that there is a natural map induced by multiplication of sections on fibers
$$\GG_m \to (g'_*\OO_{W'}(cmK_{W'/\Sigma'}+cmE'_0 ))^{**}\otimes \OO_{\Sigma'}(cA_{\Sigma'}),
$$
where the latter one is a subsheaf of
$$
g'_*\OO_{W'}(cmK_{W'/\Sigma'}+cmE'_0+E_{W'} )\otimes \OO_{\Sigma'}(cA_{\Sigma'}),
$$
for some effective $g'$-exceptional divisor  $E_{W'}$ as the non-locally-free locus is of codimension at least $2$.
Hence the global section of $\GG_m$ induces a global section of the last sheaf which does not vanish on the point $E'_0\cap C_y$ for a general $y$, that is,
an effective divisor
$$H'_m\sim cmK_{W'/\Sigma'}+cmE'_0+E_{W'} +c{g'}^*A_{\Sigma'}$$
such that $E'_0\not\subseteq \Supp(H'_m)$. Note that, by the construction, $E_{W'}$ is $\pi'$-exceptional. Pushing forward to $W$, we have
$$\pi'_*H'_m+\pi'_*{g'}^*(cm K_{\Sigma'/\bF_a}+E_{\Sigma'})\sim cmK_{W/\bF_a}+cmE_0+cg^*A.$$
Note that $cmK_{\Sigma'/\bF_a}+E_{\Sigma'}$ is an effective $\psi'$-exceptional divisor, hence $\pi'_*{g'}^*(cm K_{\Sigma'/\bF_a}+E_{\Sigma'})$ is effective and its support does not contain $E_0$ as $E_0$ dominates $\bF_a$.
Therefore we can just take $H_m:=\pi'_*H'_m+\pi'_*{g'}^*(K_{\Sigma'/\bF_a}+E_{\Sigma'})$.\end{proof}

Then we show the following inequality.
\begin{claim}\label{claim K-F pseff}
$((3\pi^*K_X-(a-6)F)\cdot {D_0}\cdot \pi^*K_X)\geq 0$.
\end{claim}
\begin{proof}
Note that
\begin{align*}
K_{W/\bF_a}+E_0{}&=(\pi^*K_X+E_\pi)+((a+2)F+2B)+E_0\\
{}&=3\pi^*K_X-(a-6)F+E_\pi+E_0-2E''.
\end{align*}
Write $E_\pi+E_0-2E''=L-N$ where $L$ and $N$ are effective $\bQ$-divisors with no common components, $L$ is $\pi$-exceptional and, clearly, the supports of $L$ and $N$ do not contain $E_0$ by coefficient computation.
Recall that $A=t_1\ell+t_2\sigma_0$ for some $t_1>at_2>0$. Now we consider
\begin{align*}
{}&H_m+cmN+ct_2E''\\
\sim{}& cmK_{W/\bF_a}+cmE_0+cg^*A+cmN+ct_2E''\\
={}&cm(3\pi^*K_X-(a-6)F+L-N)+c(t_1F+t_2B)+cmN+ct_2E''\\
\sim{}&cm(3\pi^*K_X-(a-6)F)+c(t_2\pi^*K_X+(t_1-at_2+2t_2)F)+cmL.
\end{align*}
Note that $F=\pi^*F_X$ and $L$ is effective $\pi$-exceptional, hence $cmL$ is contained in the fixed part of the above divisor. This implies that $H_m+cmN+ct_2E''-cmL$ is an effective divisor. Dividing this divisor by $cm$,  we obtain an effective $\bQ$-divisor $G_m$ such that
$$
G_m\sim_\bQ3\pi^*K_X-(a-6)F+\frac{1}{m}(t_2\pi^*K_X+(t_1-at_2+2t_2)F).
$$
Also note that $\coeff_{E_0}G_m=\frac{1}{m}\coeff_{E_0}t_2E''= \frac{t_2}{m}$ since by construction $E_0$ is not a component of $H_m$, $N$, $L$.
Now denote
\begin{align*}
\mu_m:=\coeff_{D_0}G_m=\coeff_{D_0}\Big(G_m-\frac{t_2}{m}E_0\Big).
\end{align*}
That is, we may write $G_m-\frac{t_2}{m}E_0=\mu_m D_0+G'$, where $G'$ is effective and its support does not contain $E_0$ or $D_0$.
Hence
$$\Big(\Big(G_m-\frac{t_2}{m}E_0\Big)\cdot E_0\cdot F\Big)=((\mu_m D_0+G')\cdot E_0\cdot F)\geq \mu_m.$$
Here we use the facts that $(D_0\cdot E_0\cdot F)=1$, $F$ is nef, and $(G'\cdot E_0)$ is an effective cycle. 
On the other hand, since both $G_m$ and $F$ are $\pi$-trivial and $E_0$ is $\pi$-exceptional, we have $(G_m\cdot F\cdot E_0)=0$ and hence
\begin{align*}
\Big(\Big(G_m-\frac{t_2}{m}E_0\Big)\cdot E_0\cdot F\Big)= -\frac{t_2}{m} (E_0|_F)^2.
\end{align*}
In particular, $\lim_{m\to \infty}\mu_m=0$. Hence we have
\begin{align*}
{}&((3\pi^*K_{X}-(a-6)F)\cdot {D_0}\cdot \pi^*K_X)\\=&\lim_{m\to \infty} (G_m\cdot {D_0}\cdot \pi^*K_X)\\
=&\lim_{m\to \infty} ((G_m-\mu_mD_0)\cdot {D_0}\cdot \pi^*K_X)\geq 0.\end{align*}
For the last inequality, we just use the fact that $D_0$ is not contained in the support of the effective $\bQ$-divisor $G_m-\mu_mD_0$.
\end{proof}

\noindent {\bf Sub-step 2.6.}  The main inequality for Step 2.

Note that, by Claim \ref{E0D0}(2),
\begin{align*}
{}&( (3\pi^*K_X- (a-6)F )\cdot (B-D_0)\cdot \pi^*K_X )\\= {}& (3\pi^*K_X\cdot (B-D_0)\cdot \pi^*K_X )\geq 0.
\end{align*}
Hence Claim \ref{claim K-F pseff} implies that
$$
((3\pi^*K_X- (a-6)F )\cdot B\cdot \pi^*K_X )\geq 0.
$$
Since $(F \cdot B\cdot \pi^*K_X )=(F \cdot S\cdot \pi^*K_X )=(C\cdot \pi^*K_X |_F)=1$,
hence
$$
(\pi^*K_X^2\cdot B)\geq \frac{a-6}{3}(F \cdot B\cdot \pi^*K_X )\geq \frac{a}{3}-2.
$$
Finally, we have
\begin{align*}
K_X^3={}&(\pi^*K_X^2\cdot (\pi^*K_X+2F))-2\\
\geq{}& (\pi^*K_X^2\cdot (aF+B))-2\\
\geq{}&\frac{4a}{3}-4=\frac{4}{3}p_g(X)-\frac{8}{3}.
\end{align*}
The proof is completed.
\end{proof}

\begin{remark}
Let us give a historical remark on the proof of the last case (Theorem \ref{last}). Assuming $X$ is Gorenstein, there are two approaches to the last case in history, one as in \cite{MChen04JMSJ} and one as in \cite{MChen04MRL, CCZ06}. Both approaches use the fact that $\Mov|K_X|$ is base point free, which automatically holds as  $X$ is Gorenstein (Corollary \ref{base12}(3)). The latter one gives a better inequality by considering the linear system $|2K_X|$, but it essentially uses the fact that $h^0(2K_X)=\frac{1}{2}K_X^3-3\chi(\OO_X)$ which holds if and only if $X$ is Gorenstein, and thus not available in our setting. Here our proof (up to Step 1) follows the former approach, and shows firstly a weaker inequality $K_X^3\geq \frac{4}{3}p_g(X)-\frac{14}{3}$ as in \cite{MChen04JMSJ}. So, in the end, we have to put our effort in Step 2 to improve this weaker inequality by analyze exceptional cases which is not known before. This involves special geometry of $(1,2)$-surfaces and applying  weak positivity of direct images to refine the estimate of intersection numbers.
\end{remark}

\subsection{Proof of main theorems}\

\begin{proof}[Proof of Theorem \ref{main2}]
Let $X$ be a minimal projective $3$-fold of general type. If $p_g(X)\leq 4$, the inequality follows from \cite[Theorem 1.5]{MChen07}. We may always assume that $p_g(X)\geq 5$.
 We may always consider the non-trivial canonical map $\varphi_1=\Phi_{|K_X|}$ defined by the canonical linear system $|K_X|$. Set $d_X:=\dim \overline{\varphi_1(X)}$.

If $d_X=3$, then $K_X^3 \ge 2p_g(X)-6\ge \frac{4}{3}p_g(X)-\frac{10}{3}$ by \cite{Kob92} or \cite[Proposition 3.1]{MChen04JMSJ}.

If $d_X=2$, then the inequality follows from Theorems \ref{d2g3} and \ref{d2g2}.

If $d_X=1$ and $|K_X|$ is not composed with a rational pencil of $(1,2)$-surfaces, then either $|K_X|$ is not composed with a pencil of $(1,2)$-surfaces, or $|K_X|$ is composed with an irrational pencil of $(1,2)$-surfaces. The former case is done by Theorems \ref{not 11 not 12} and \ref{1,1} with $p_g(X)\geq 5$, and the latter case follows from Corollary \ref{base12}(1) and Theorem \ref{last} after replacing $X$ with a minimal model $Y$ with $\Mov|K_Y|$ free.

If $d_X=1$, $|K_X|$ is composed with a rational pencil of $(1,2)$-surfaces, and $p_g(X)\geq 21$, then the inequality follows from Corollary \ref{base12}(2) and Theorem \ref{last} after replacing $X$ with a minimal model $Y$ with $\Mov|K_Y|$ free.
\end{proof}

\begin{proof}[Proof of Theorem \ref{main}]
Let $X$ be a projective $3$-fold of general type. Take a resolution $W\to X$ and take $Y$ to be a minimal model of $W$, which is a minimal projective $3$-fold of general type birational to $X$. It is clear that $p_g(X)=p_g(W)=p_g(Y)$ and $\vol(X)=\vol(W)=K_Y^3$. Hence the inequality follows from Theorem \ref{main2}.
\end{proof}

\appendix

\section{\bf Surfaces with $K^2=1$ and  $p_g=2$}\label{appendix}

\begin{center}
by J\'{a}nos Koll\'{a}r\\
{\small Department of Mathematics, Fine Hall, Washington Road\\
Princeton University,  Princeton, NJ 08544-1000  USA\\
e-mail:  kollar@math.princeton.edu}
\end{center}

\medskip

Given a variety $X$ and a divisor class $H$,
it is an interesting general problem to compute
$\lct(X;H)$, which is the infimum of the $\lct(X;\Delta)$ where
 $\Delta$ is an effective $\q$-divisor such that $\Delta\equiv H$.
This problem has received a lot of attention when $H=-K_X$ is ample, see for example \cite{Chel08, PW10, PW11, CK14}.
Here we  are mainly interested in the case when $S$ is a  surface with $K^2=1$, $p_g=2$ and
$H=K_S$.

\begin{thm} \label{lcth.from.h0.cor}
Let $S$ be a projective surface with Du~Val singularities. Assume that $K_S$ is  ample, $(K_S^2)=1$ and $h^0\bigl(S, \OO_S(2K_S)\bigr)=4$. Let $\Delta_S$ be an effective $\q$-divisor such that $\Delta_S\equiv K_S$. Then $\lct(S;\Delta_S)\geq \tfrac1{10}$.
\end{thm}

The proof has 3 parts. First, in Proposition~\ref{lcth.from.h0.prop} we estimate a global $H^0$ from below using a local invariant, called
 the {\it minimal multiplier codimension}
$$
\operatorname{mcd}(c):=\min_{G, \Delta} \Bigl\{\dim\Bigl(\bigl(\cc[x,y]/{\mathcal J}^+(\Delta)\bigr)^G\Bigr)\Bigr\},
\eqno{(\ref{lcth.from.h0.cor}.1)}
$$
where  $G$ runs though all finite subgroups of $\SL_2(\cc)$,   $\Delta$ runs through all $G$-invariant divisors such that $\lct_0(\cc^2; \Delta)<c$ and
we use the upper multiplier ideal
${\mathcal J}^+(\Delta):= {\mathcal J}((1-\epsilon)\Delta)$
for $0<\epsilon\ll 1$ \cite[9.2.1]{laz-book}. Then we compute  $\operatorname{mcd}(c) $ in  Proposition~\ref{mcd.prop}.
Finally we need to look more carefully at the case when $S$ has an $E_8$ singularity.

\begin{exmple} Consider the pair
$$
S:=\bigl(x^7y^3+y^{10}+z^5+t^2=0\bigr)\subset \p(1,1,2,5)
\qtq{and}  \Delta:= (y=0).
$$
It is easy to check that $S$ has a unique singular point, at $(1{:}0{:}0{:}0)$, and it has type $E_8$. Thus
$S$ is a projective surface with Du~Val singularities, $K_S=\OO_S(1)$ is  ample, $(K_S^2)=1$ and  $h^0\bigl(S, \OO_S(2K_S)\bigr)=4$. Furthermore,
$\lct(S;\Delta)= \tfrac1{10}$. The latter can be checked directly but it is easiest to see using the local universal cover at the singularity as in the proof of Theorem~\ref{lcth.from.h0.cor} below.
\end{exmple}

\begin{prop} \label{lcth.from.h0.prop}
Let $S$ be a projective surface with Du~Val singularities and $H$ an ample Cartier divisor such that $(H^2)=1$. Let $\Delta_S$ be an effective $\q$-divisor such that $\Delta_S\equiv H$. Then
$h^0\bigl(S, \OO_S(K_S+H)\bigr)\geq \operatorname{mcd}\bigl(\lct(S;\Delta_S)\bigr)$.
\end{prop}

\begin{proof}If $\Delta_S$ contains a curve $B$ with coefficient $b$ then
$b\leq b(B\cdot H)\leq (H^2)$. Thus $\rdown{(1-\epsilon)\Delta_S}=0$ and
$\bigl(S, (1-\epsilon)\Delta_S\bigr)$  has only isolated non-log-canonical centers.
Writing
$K_S+H\equiv K_S+ \epsilon H+ (1-\epsilon)\Delta_S$,
Nadel's vanishing  \cite[9.4.17]{laz-book}
 gives a surjection
$$
H^0\bigl(S, \OO_S(K_S+H)\bigr)\onto  H^0\bigl(S,  \OO_S(K_S)\otimes \OO_S/   {\mathcal J}^+(\Delta_S)\bigr),
\eqno{(\ref{lcth.from.h0.prop}.1)}
$$
and the right hand side has dimension $=\dim \OO_S/   {\mathcal J}^+(\Delta_S)$.
The computation of $\OO_S/   {\mathcal J}^+(\Delta_S) $ is local in the Euclidean topology.
We can thus write  $(s\in S,\Delta_S)$ as $(0\in \cc^2,\Delta)/G$ for some $G\subset \SL_2$ where $\Delta$ is the pullback of $\Delta_S$ through the quotient map.
By \cite[9.5.42]{laz-book}, $\OO_S/{\mathcal J}^+(\Delta_S)$ is
 the $G$-invariant part of
$\cc[x,y]/{\mathcal J}^+(\Delta)$. Thus
$$
h^0\bigl(S, \OO_S(K_S+H)\bigr)\geq \dim\Bigl(\bigl(\cc[x,y]/{\mathcal J}^+(\Delta)\bigr)^G\Bigr).
\eqno{(\ref{lcth.from.h0.prop}.2)}
$$
\end{proof}

\begin{prop} \label{mcd.prop} The inverse function of $\operatorname{mcd}$ is computed by the following table.
$$
\begin{array}{rccccccccccc}
\operatorname{mcd}(c)=   &0&1 & 2 & 3 & 4 & 5 & 6 & n=7{-}9 & n=10{-}13 &  n\geq 14\\[1ex]
c\geq &1 & \tfrac1{7}  &  \tfrac1{11} & \tfrac1{13} & \tfrac1{16} &\tfrac1{17} & \tfrac1{19} & \tfrac1{n+14}  & \tfrac1{n+15}  &  \tfrac1{2n+1}
\end{array}
$$
\end{prop}

\begin{proof}
We need to
study pairs  $(\cc^2, \Delta)$ where $\Delta=d\cdot \bigl(g(x,y)=0\bigr)$ is a $G$-invariant divisor such  that $\lct_0(\cc^2; \Delta)<c$.
We distinguish 2 cases.

(Abelian case.)  The argument in \cite[6.40]{ksc}  (going back to Varchenko) proves   that  $\lct_0(\cc^2;\Delta)$ is computed by suitable weighted coordinates. The proof also works equivariantly for  abelian group actions.
So in suitable (local analytic) coordinates $x,y$ and weights  $w_x, w_y$,
$$
g(x,y)\in \bigl( x^iy^j:   iw_x+jw_y> \tfrac2{cd}\cdot \tfrac{w_x+w_y}{2}\bigr).
\eqno{(\ref{mcd.prop}.1)}
$$
Thus, by \cite[9.3.27]{laz-book},
$$
{\mathcal J}^+(\Delta)\subset \bigl( x^iy^j:   iw_x+jw_y\geq \bigl(\rdown{\tfrac2{c}}-1\bigr) \cdot \tfrac{w_x+w_y}{2}\bigr).
\eqno{(\ref{mcd.prop}.2)}
$$
 Since $xy$ is $G$-invariant for $A_n$, we get invariants $(xy)^i$ for  $i\leq \tfrac1{c}-2$  (and other invariants if $n$ is small).
There are always more $G$-invariants  than in the non-abelian case.

(Non-abelian case.) By Lemma~\ref{D.E.lem} we have
 $\lct_0(\cc^2; \Delta)=2/(\mult_0\Delta)$ and in
(\ref{mcd.prop}.1) we can take $w_x=w_y=1$.
Thus we obtain that if $r$ is an integer and  $\lct_0(\cc^2; \Delta)<\tfrac 1{r}$ then
${\mathcal J}^+(\Delta)\subset (x,y)^{2r-1}$. Hence if  $\tfrac1{r+1}\leq c<\tfrac 1{r}$ then
$$
\operatorname{mcd}(c):=\min_{G} \Bigl\{\dim\Bigl(\bigl(\cc[x,y]/(x,y)^{2r-1}\bigr)^G\Bigr)\Bigr\},
\eqno{(\ref{mcd.prop}.3)}
$$
where  $G$ runs though all non-abelian finite subgroups of $\SL_2(\cc)$.
We thus need to compute the dimension of the space  of $G$-invariant polynomials in $\cc[x,y]/(x,y)^{2r-1}$ as a function of $r$ and $G$.
It has been known at least since  Felix Klein
 that the ring of $G$-invariants
has 3 homogeneous  generators.
The following table lists their degrees; see \cite[Secs.34--39]{MR0169108}.
$$
\begin{array}{cc}
\mbox{singularity} &  \mbox{degrees of generators of invariants}\\
D_n & 4, 2n{-}4, 2n{-}2\\
E_6 & 6,8,12\\
E_7 & 8,12,18 \\
E_8 &  12,20,30 \\
\end{array}
\eqno{(\ref{mcd.prop}.4)}
$$
A quick hand computation shows
 that $E_8$ has the fewest invariants for degrees $\leq 56$ and
the entries for $n=0,\dots, 13$ are computed as follows.
Let $I:=I_{120}$ denote the binary icosahedral group acting linearly on $\cc[x,y]$. Pick $m\leq 56$ such that there is an $I$-invariant of degree $m$. This gives a threshold value of $\frac2{m+2}$ and above it we put  the dimension of the space  of $I$-invariant polynomials of degree $\leq m$.

The  entries for $n\geq 14$ are computed from the binary dihedral groups
$D_r$ for  $r> 2n+2$, whose only invariants of degree $<4n$ are the  powers of  $(xy)^2$.
\end{proof}

\begin{say}[Proof of Theorem~\ref{lcth.from.h0.cor}]
We see that  $\operatorname{mcd}\bigl(\lct(S;\Delta_S)\bigr)\leq 4$  by
Proposition~\ref{mcd.prop}. For $E_8$ we have invariants in degrees
$0,12,20,24,30, \dots$, thus (\ref{mcd.prop}.3)
 gives that $\lct(S;\Delta_S)\geq\tfrac1{16}$.
The next worst case is $E_7$ with invariants in degrees
$0,8,12, 16, 18, \dots$ and then  $\lct(S;\Delta_S)\geq\tfrac1{10}$, as needed.

It remains to look in more detail at the $E_8$ cases.
Again we work with the cover $\pi:(0,\cc^2)\to (s,S)$.
For a curve $s\in C\subset S$ let $\tilde C\subset \cc^2$ denote the preimage of $C$.

If $\lct(S;\Delta_S)<\tfrac17$ then (\ref{lcth.from.h0.prop}.1) gives a surjection
$$
H^0\bigl(S, \OO_S(2K_S)\bigr)\onto \bigl(\cc[x,y]/(x,y)^{13}\bigr)^G.
$$
Thus we get a curve $C\in |2K_S|$ such that $\mult_0\tilde C=12$.
Write any $\Delta_S$ as $\tfrac12 \lambda C+(1-\lambda)\Delta'_S$ where
$\Delta'_S\equiv K_S$ and
 $\Supp(\Delta'_S)$ does not contain $C$. Then we get that
$$
\mult_0\Delta'\cdot  \mult_0\tilde C\leq
\bigl(\Delta'\cdot  \tilde C\bigr)_0\leq 120(\Delta'_S\cdot   C)= 240.
$$
So
$$\mult_0\Delta=\tfrac{\lambda}{2}\cdot\mult_0 \tilde{C}+(1-\lambda)\mult_0\Delta'\leq 20$$
and  $\lct(S;\Delta_S)=\lct_0(\cc^2;\Delta)\geq \tfrac1{10}$ by Lemma~\ref{D.E.lem}.
\qed
\end{say}

The following is
closely related to the notion of weakly-exceptional singularities introduced in  \cite[Sec.5]{sho-3ff}.   

\begin{lem} \label{D.E.lem}
 Let $G\subset \GL_2$ be a non-abelian finite subgroup and
$\Delta$ a $G$-invariant $\q$-divisor on $\cc^2$. Then
$(\cc^2, \Delta)$ is log canonical at $0$ iff $\mult_0\Delta\leq 2$.
\end{lem}

\begin{proof}
 The non-trivial claim is that if  $\mult_0\Delta= 2$ then $(\cc^2, \Delta)$ is log canonical at $0$. To see this blow up the origin to get
$\pi:S\to \cc^2$ with exceptional curve $E$. Then we have
$K_S+E+\Delta'\simr \pi^*\bigl(K_{\cc^2}+\Delta\bigr)$ and $\Delta'|_E$ has degree 2. It is also $G$-invariant and $G$ has no fixed points on $E\cong \p^1$. So each point appears in $\Delta'|_E$ with coefficient $\leq 1$ hence
$\bigl(E, \Delta'|_E\bigr)$ is  log canonical. Thus $(S, E+\Delta')$ is  also log canonical by  adjunction \cite[4.9]{kk-singbook}. \end{proof}

It is remarked by one referee that  the non-equivariant analogue of the above lemma is the following: if
$(S, \Delta)$ is a log pair, and $P$ is a du Val singular point of $S$ that is not of type A, then
$(S, \Delta)$  is log canonical at $P$ if and only if the discrepancy $a(E; S, \Delta) \geq -1$, where $E$ is the $(-2)$-curve on the minimal resolution intersecting three other $(-2)$-curves.

\begin{remark} 1. While the values in Proposition~\ref{mcd.prop} are optimal,
it is not clear how sharp  the bounds in Proposition~\ref{lcth.from.h0.prop} are. It is quite likely that the methods of  Theorem~\ref{lcth.from.h0.cor} can be used to improve them in all cases.
For example, if  $S$ is smooth then we get  much better bounds:
$$
h^0\bigl(S, \OO_S(K_S+H)\bigr)<\tbinom{m}{2} \quad\Rightarrow \quad
\lct(S;\Delta)\geq \tfrac2{m},
$$
so the $1/(2n+1)$ in Proposition \ref{lcth.from.h0.prop}
is replaced roughly by  $\sqrt{2/n}$.

2. The assumption $(H^2)=1$ made the above arguments much simpler.
Without it, for  $n:=h^0(S, K_S+H)\geq 14$ we get
$$
\lct(S;\Delta)\geq \tfrac{1}{(H^2)}\cdot  \tfrac{1}{2n+1},
$$
but it is likely that the $(H^2)$ factor is not needed.
\end{remark}

\section*{\bf Acknowledgment}
We would like to thank Fabrizio Catanese, Christopher D. Hacon, Yong Hu, Yujiro Kawamata, J\'anos Koll\'ar, Yongnam Lee, and De-Qi Zhang for effective discussions and comments during the long period of the preparation of this paper.
Part of this paper was written during C. Jiang's visit,  in August and December 2017, to Fudan University for which he is grateful for the support and hospitality.  The second author is partially supported as a member of LMNS, Fudan University. We are grateful to  J\'anos Koll\'ar who is so kind to send us his note (Appendix A), which greatly improves and simplifies the proof of ``$\glct\geq 1/13$'' in the earlier draft of this paper. We are grateful to the referees for many useful comments and suggestions.

\end{document}